\documentclass[12pt,reqno,a4paper]{article}

\usepackage{euscript}
\usepackage{amsmath,amssymb,amsfonts,amsthm}
\usepackage{mathtools}
\usepackage{fullpage}
\usepackage{xcolor}
\usepackage{dsfont}
\usepackage[colorlinks=true, linkcolor=blue]{hyperref}
\usepackage{sectsty}

\newtheoremstyle{special}%
{}%
{}%
{}%
{}%
{\scshape}%
{.}%
{.5em}%
{}

\newtheorem{theorem}{Theorem}
\newtheorem{proposition}[theorem]{Proposition}
\newtheorem{lemma}[theorem]{Lemma}
\newtheorem{cor}[theorem]{Corollary}
\newtheorem{definition}[theorem]{Definition}

\theoremstyle{special}
\newtheorem{remark}[theorem]{Remark}
\newtheorem{example}[theorem]{Example}

\renewcommand{\epsilon}{\varepsilon}
\renewcommand{\L}{\mathcal{L}}

\allsectionsfont{\centering\mdseries\normalsize\scshape}

\def\Id{\text{\rm Id}}

\def\N{\mathbb{N}}
\def\Z{\mathbb{Z}}
\def\R{\mathbb{R}}

\DeclareMathOperator{\var}{Var}

\begin{document}

	\title{Quenched linear response for smooth expanding on average cocycles}

	\author{Davor Dragi\v cevi\'c\footnote{Department of Mathematics, University of Rijeka, Radmile Matej\v ci\' c 2, 51000 Rijeka,  Croatia. \texttt{Email:ddragicevic@math.uniri.hr}} , Paolo Giulietti\footnote{Dipartimento di Matematica, Università di Pisa, Largo Bruno Pontecorvo 5, 56127 Pisa, Italia \texttt{Email:paolo.giulietti@unipi.it}} , Julien Sedro\footnote{Sorbonne Universit\'e and Universit\'e de Paris, CNRS, Laboratoire de Probabilit\'es, Statistique et Mod\'elisation, F-75005 Paris, France \texttt{Email:sedro@lpsm.paris}}}

	\maketitle
	\abstract{We establish an abstract quenched linear response result for random dynamical systems, which we then apply to the case of smooth expanding on average cocycles on the unit circle. In sharp contrast to the existing results in the literature, we deal with the class of random dynamics that does not necessarily exhibit uniform decay of correlations. Our techniques rely on the infinite-dimensional ergodic theory and in particular, on the study of the top Oseledets space of a parametrized transfer operator cocycle. Finally, we exhibit a surprising phenomenon: a random system and a smooth observable for which quenched linear response holds, but annealed response fails.}
	
	\section{Introduction}
	
	\subsection{Linear response for deterministic and random dynamical systems}
	Let $M$ be a (compact) Riemannian manifold  and consider a family $(T_{\omega, \epsilon})$ of sufficiently smooth maps acting on $M$ and indexed by $\omega \in \Omega$ and $\epsilon \in I$, where $(\Omega, \mathcal F, \mathbb P)$ is a probability space and $0\in I\subset \R$ is an interval. One can view $T_{\omega, \epsilon}$ as a  `small' perturbations of $T_{\omega, 0}$.
	Endowing the probability space $\Omega$
	with an invertible map $\sigma \colon \Omega \to \Omega$ that is measure-preserving and ergodic, we may form the \emph{random products} over $\sigma$, defined by 
	\[
	T_{\omega, \epsilon}^n:=T_{\sigma^{n-1}\omega, \epsilon} \circ \ldots \circ T_{\sigma \omega, \epsilon} \circ T_{\omega, \epsilon}.
	\]
	Let us assume that for each $\epsilon \in I$, the cocycle $(T_{\omega, \epsilon})_{\omega \in \Omega}$ admits a unique physical equivariant measure, that is, a measurable family of probability measures $(h_{\omega, \epsilon})_{\omega \in \Omega}$, 
	such that 
	\[
	T_{\omega, \epsilon}^*h_{\omega, \epsilon}=h_{\sigma \omega, \epsilon} \quad \text{for $\mathbb P$-a.e. $\omega \in \Omega$,}
	\]
	where $T_{\omega, \epsilon}^*h_{\omega, \epsilon}$ denotes the push-forward of $h_{\omega, \epsilon}$ with respect to $T_{\omega, \epsilon}$, and such that the ergodic basin of $h_{\omega, \epsilon}$ has a positive Riemannian volume.
	Then, it is natural to ask the following questions: is the map $\epsilon \mapsto h_{\omega, \epsilon}$
	differentiable at $\epsilon=0$ (in a suitable sense)? If so,  can one
	give an explicit formula for its derivative?
	These questions are known as the \emph{linear response} problem.
	
	We emphasize that in the context of \emph{deterministic} dynamical systems (which corresponds, in our setting, to the case when $\Omega$ is a singleton), the linear response problem has been thoroughly studied: it was discussed for smooth expanding 
	systems (either on the unit circle or in higher dimensions)~\cite{B2,Baladibook,S}, piecewise expanding maps of the interval \cite{B1,BS1}, unimodal maps~\cite{BS2}, intermittent maps \cite{BahSau,BT,K}, hyperbolic diffeomorphisms and flows~\cite{BL1, BL2, GL, Ruelle97}, as well as for large classes of partially hyperbolic systems~\cite{Dol}. We refer to~\cite{B2} for a detailed survey of the linear response theory for deterministic dynamical systems. 
	
	The setting of random dynamical systems can be divided in two different subcases, the annealed case and
	the quenched case. The annealed case
	may be studied by methods very similar to the deterministic case, via weak spectral
	perturbation for the associated family of transfer operators, or quantitative stability statements for fixed points of Markov operators, and often enjoy a convenient
	`regularization property': we refer to~\cite{BRS,GG,GS,GL,HM} for details.

	Quenched statistical stability (i.e. continuity, in a suitable sense, of the map $\epsilon\mapsto h_{\omega,\epsilon}$ in $\epsilon=0$) has been studied for some time: see, e.g., \cite{Bo} for random subshift of finite type, \cite{BKS} for smooth expanding maps, \cite{HC,DS} for Anosov systems. Closer to the focus of the present paper, quenched statistical stability for an expanding on average cocycles of piecewise expanding systems, exhibiting non-uniform decay of correlations (as considered by Buzzi \cite{Buzzi}) was established in \cite{FGTQ}.
	
	On the other hand, quenched linear response has begun to receive adequate attention only very recently. More precisely, the quenched linear response for (smooth) random dynamical systems was discussed in~\cite{RS}  for expanding dynamics,  in~\cite{DS1} for hyperbolic dynamics, and finally in~\cite{CN} for some classes of partially hyperbolic dynamics. 
	
	A common feature of all those results is that they are restricted to 
	the case when  the `unperturbed'  cocycle $(T_{\omega,0})_{\omega \in \Omega}$   exhibits
	uniform (with respect to   $\omega$) decay of correlations (see~\cite[Remark 4.20]{RS}, \cite[eq. (25)]{DS1}, \cite[Definition 3.4]{CN} and~\cite[eq. (QR0)]{CN}).  
	In addition, various other assumptions, such as those on appropriate Lasota-Yorke inequalities are of `uniform'-type, i.e. the  associated constants are not allowed to depend on the random parameter $\omega$. 
	
	However, there are many interesting classes of random dynamical systems which, in general, do not exhibit uniform decay of correlations. Those include smooth or piecewise smooth random expanding on average maps~\cite{Buzzi, K1}, as well as random distance expanding maps~\cite{HK}. For some recent results dealing with limit laws for such systems, we refer to~\cite{DS,DHS, HK} and references therein.
	
	\subsection{Contributions of the present paper}
	The main objective of the present paper is to establish a linear response result (see Theorem~\ref{LR1}) for the so-called parameterized smooth expanding on average cocycles on the unit circle (see Definition~\ref{EC}). For this purpose, we formulate abstract statistical stability and linear response results for random dynamical systems (see Theorem~\ref{SS} and \ref{LR}), and (see Section~\ref{PSC}) verify all of their assumptions in the case of parameterized smooth expanding on average cocycles. In sharp contrast with the previously discussed results in~\cite{CN, DS, RS}, our approach \emph{does not require stochastic uniformity} (i.e. uniformity w.r.t $\omega$) and deal with systems exhibiting nonuniform decay of correlations. We also note that our Theorem \ref{SS} allows to recover the statistical stability results of \cite{FGTQ}, for a class of piecewise monotone, expanding on average random systems introduced by Buzzi \cite{Buzzi}, which is, to the best of our knowledge, the only result dealing with quenched stability in a stochastically non-uniform setting. 
	
	Our methods rely on the infinite-dimensional version of the multiplicative ergodic theorem (MET) established in~\cite{GTQ}, coupled with some perturbative estimates on transfer operators generalizing those of \cite{GG,GS} and close to those in~\cite{DS1}. In particular, our approach requires a detailed analysis 
	of the so-called top Oseledets space of a parameterized transfer operator cocycle  (see for example Corollary~\ref{KORO}). We stress that unlike~\cite{CN,DS1}, we do not use the so-called Mather operator, since in our setting it is unclear on which space would this operator act on (due to nonuniform behavior).
	
	The reader will notice that our linear response results (Theorems~\ref{LR} and~\ref{LR1}) require a suitable  discretization of the parameter $\epsilon$. The explanation behind this is the following: for each $\epsilon \in I$, the set of random parameters $\omega$
	for which certain estimates (as well as the existence) on the equivariant measure $h_{\omega, \epsilon}$ hold is some set $\Omega_\epsilon$ with full probability. Hence, in order to study the behavior of $h_{\omega, \epsilon}$ when $\epsilon \to 0$, we would need to deal with the set $\bigcap_{\epsilon \in I} \Omega_\epsilon$, which could fail to even be measurable. Our statements of Theorems~\ref{LR} and~\ref{LR1} are tailored to overcome this subtle issue. We refer to~\cite[p.3]{DS1} for additional discussion.
	
	Finally, in an appendix, we present an example of a surprising phenomenon: a random system and an observable such that our Theorem \ref{LR1} applies, so that quenched response holds; but the annealed formula does not converge, implying that annealed linear response cannot hold. We hope that this example convince the reader of the richness of the class of examples considered in the present paper.
	
	\subsection{Organization of the paper}
	In Section~\ref{prel}, we recall some  basic material from the infinite-dimensional ergodic theory, which will be used in the paper. In Section~\ref{ABS}, we formulate our abstract version of the linear response result for random dynamical systems (see Theorem~\ref{LR}). Finally, in Section~\ref{PSC}, we introduce and study parameterized smooth expanding on average cocycles on the unit circle. We establish several auxiliary results whose aim is to verify that all assumptions of Theorem~\ref{LR} are satisfied in this setting. We conclude by establishing  an explicit linear response result for parameterized smooth expanding on average cocycles (see Theorem~\ref{LR1}), as a corollary of Theorem~\ref{LR}.
	\\ In the Appendix, we present an example of a random system and a smooth observable for which quenched response holds, but annealed response does not.

	\section{Preliminaries from multiplicative ergodic theory}\label{prel}
	The purpose of this section is to recall basic notions of the multiplicative ergodic theory that will be used throughout this paper.  Our presentation follows closely~\cite[Section 2.1]{DS}.
	
	We begin by recalling the notion of a linear cocycle. 
	\begin{definition}
		A tuple $\mathcal{R}=(\Omega, \mathcal{F}, \mathbb P, \sigma, \mathcal B, \L)$ is said to be a \emph{linear cocycle} or simply a \emph{cocycle} if the following conditions hold:
		\begin{itemize}
			\item $(\Omega,\mathcal F,\mathbb P)$ is a probability space and  $\sigma \colon \Omega \to \Omega$ is an invertible and ergodic $\mathbb P$-preserving transformation;
			\item $(\mathcal B, \|\cdot \|)$ is a Banach space and
			$\mathcal L\colon \Omega\to L(\mathcal B)$ is a family of bounded linear operators.
		\end{itemize}
		We say that $\mathcal L$ is the \emph{generator} of $\mathcal R$.
	\end{definition}

	We will often identify a cocycle $\mathcal R$ with its generator $\mathcal L$. 	Moreover, we will write $\mathcal L_\omega$ instead of $\mathcal L(\omega)$.
	\\Recall that a cocycle $\mathcal R$ is said to be \emph{strongly measurable} if $\Omega$ is a Borel subset of a separable, complete metric space, $\sigma$ is a homeomorphism and
	$\omega \mapsto \L_\omega h$ is measurable for each $h\in \mathcal B$.
	\\Set, for $\omega \in \Omega$ and $n\in \N$:
	\[
	\L_\omega^n:=\L_{\sigma^{n-1} \omega} \circ \ldots \circ  \L_{\sigma \omega} \circ \L_\omega.
	\]
	If $\L$ is strongly measurable and such that
	\begin{equation}\label{int}
		\int_\Omega \log^+ \|\L_\omega \|\, d\mathbb P(\omega) <+\infty,
	\end{equation}
	where $\|\L_\omega\|$ denotes the operator norm of $\L_\omega$, it follows from Kingman's subadditive ergodic theorem that, for $\mathbb P$-a.e. $\omega \in \Omega$, the following limits exist:
	\[
	\Lambda(\mathcal{R}) := \lim_{n\to \infty} \frac1n\log \| \L_\omega^n \|
	\]
	and 
	\[
	\kappa(\mathcal{R}): =  \lim_{n\to \infty} \frac1n\log ic( \L_\omega^n),
	\]
	where \[\text{ic}(A):=\inf\Big\{r>0 : \ A(B_{\mathcal B})~\text{admits a finite covering by balls of radius }r \Big\},\] $B_{\mathcal B}$ is the unit ball of $\mathcal B$, and 
	\[-\infty \le \kappa(\mathcal R) \le \Lambda (\mathcal R)<+\infty. \]
	We recall that $\Lambda (\mathcal R)$ and $\kappa(\mathcal R)$ are called the \emph{top Lyapunov exponent} and \emph{index of the compactness} of $\mathcal R$ respectively. 
	\begin{definition}
		Let $\mathcal{R}$ be a strongly measurable cocycle such that \eqref{int} holds. 
		We say that $\mathcal R$ is \emph{quasi-compact} if $\kappa (\mathcal R)<\Lambda (\mathcal R)$.
	\end{definition}
	The following result gives sufficient conditions under which a cocycle is quasi-compact (see~\cite[Lemma 2.1]{D}).
	\begin{lemma}\label{QC}
		Let $\mathcal{R}=(\Omega, \mathcal{F}, \mathbb P, \sigma, \mathcal B, \L)$ be a strongly measurable cocycle such that \eqref{int} holds. Furthermore, let  $(\mathcal B', |\cdot|)$ be a Banach space such that $\mathcal B\subset \mathcal B'$ and that the inclusion
		$(\mathcal B, \|\cdot\|) \hookrightarrow (\mathcal B',|\cdot|)$ is compact. Finally, assume the following:
		\begin{itemize}
			\item $\L_\omega$ can be extended continuously to $(\mathcal B', |\cdot|)$ for $\mathbb P$-a.e. $\omega \in \Omega$;
			\item there are  measurable functions $\alpha_\omega, \beta_\omega, \gamma_\omega: \Omega \to \R$ such that for $f\in \mathcal B$ and $\mathbb P$-a.e. $\omega \in \Omega$,
			\[
			\| \L_\omega f\| \leq \alpha_\omega \|f\| + \beta_\omega |f|
			\]
			and
			\[
			\| \L_\omega \|  \leq \gamma_\omega;
			\]
			\item we have that 
			\[
			\int_\Omega \log \alpha_\omega \, d\mathbb P(\omega) < \Lambda(\mathcal {R}) \text{ and } \int_\Omega \log \gamma_\omega \, d\mathbb P(\omega)<\infty.
			\]
		\end{itemize}
		Then, \[\kappa(\mathcal R) \le \int_\Omega \log \alpha_\omega \, d\mathbb P(\omega). \]  In particular, 
		$\mathcal R$ is quasi-compact. 
	\end{lemma}
	We are now in a position to recall the version of the multiplicative ergodic theorem (MET) established in~\cite{GTQ}.
	\begin{theorem}\label{MET}
		Let $\mathcal R=(\Omega,\mathcal F,\mathbb
		P,\sigma,\mathcal B,\mathcal L)$ be a  quasi-compact strongly measurable cocyle such that $\mathcal B$ is separable. Then, there exists $1\le l\le \infty$ and a sequence of exceptional
		Lyapunov exponents
		\[ \Lambda(\mathcal R)=\lambda_1>\lambda_2>\ldots>\lambda_l>\kappa(\mathcal R) \quad \text{(if $1\le l<\infty$)}\]
		or  \[ \Lambda(\mathcal R)=\lambda_1>\lambda_2>\ldots \quad \text{and} \quad \lim_{n\to\infty} \lambda_n=\kappa(\mathcal R) \quad \text{(if $l=\infty$).} \]
		Furthermore,  for $\mathbb P$-a.e. $\omega \in \Omega$ there exists a unique splitting (called the \textit{Oseledets splitting}) of $\mathcal B$ into closed subspaces
		\begin{equation}\label{eq:splitting}
			\mathcal B=V(\omega)\oplus\bigoplus_{j=1}^l Y_j(\omega),
		\end{equation}
		depending measurably on $\omega$ and  such that:
		\begin{enumerate} 
			\item  For each $1\leq j \leq l$, $Y_j(\omega)$ is finite-dimensional (i.e. $m_j:=\dim Y_j(\omega)<\infty$),  $Y_j$ is equivariant i.e. $\L_\omega Y_j(\omega)= Y_j(\sigma\omega)$ and for every $y\in Y_j(\omega)\setminus\{0\}$,
			\[\lim_{n\to\infty}\frac 1n\log\|\mathcal L_\omega^ny\|=\lambda_j.\]
			\item
			$V$ is weakly equivariant i.e. $\L_\omega V(\omega)\subseteq V(\sigma\omega)$ and
			for every $v\in V(\omega)$, \[\lim_{n\to\infty}\frac 1n\log\|\mathcal
			L_\omega^nv\|\le \kappa(\mathcal R).\]
		\end{enumerate}
	\end{theorem}

	\section{An abstract quenched linear response result for random dynamics}\label{ABS}
	The purpose of this section is to establish an abstract linear response result for random dynamics
	that in principle could be applied in a variety of situations.  However, our result is tailored to be applicable in the case of smooth expanding on average cocycles, which will be discussed in Section~\ref{PSC}.
	
	Let $(\Omega, \mathcal F, \mathbb P)$ be a probability space and let $\sigma \colon \Omega \to \Omega$ be an arbitrary invertible $\mathbb P$-preserving transformation. Moreover, let us assume that $\sigma$ is ergodic. We recall that a random variable $K\colon \Omega \to (0, \infty)$ is said to be \emph{tempered} if 
	\[
	\lim_{n\to \pm \infty} \frac 1 n \log K(\sigma^n \omega)=0, \quad \text{for $\mathbb P$-a.e. $\omega \in \Omega$.}
	\]
	By $\mathcal K$ we denote the set of all tempered random variables $K\colon \Omega \to (0, \infty)$. 
	\begin{remark}\label{TR}
		It is straightforward to verify that for $K_1, K_2\in \mathcal K$, we have that $K_1+K_2$, $K_1\cdot K_2$ and $\max \{K_1, K_2\}$ also belong to $\mathcal K$. Moreover,  as a simple consequence of Birkhoff's ergodic theorem, we conclude that each random variable  $K\colon \Omega \to (0, \infty)$ such that 
		$\log K\in L^1(\Omega, \mathbb P)$, belongs to $\mathcal K$. We will use these properties often throughout this paper,  and mostly without explicitly referring to this remark.
	\end{remark}
	
	In addition,  we will often use the following well-known result (see~\cite[Proposition 4.3.3]{Arnold}).
	\begin{proposition}\label{temp}
		Let $K\in \mathcal K$. Then, for each $a>0$ there exists $K_a\in \mathcal K$ such that 
		\begin{equation}\label{sw}
			K(\omega) \le K_a(\omega) \quad \text{and} \quad K_a(\sigma^n \omega) \le K_a(\omega)e^{a |n|}, 
		\end{equation}
		for $\mathbb P$-a.e. $\omega \in \Omega$ and $n\in \Z$.
	\end{proposition}
	
	\begin{remark}For the convenience of the reader, we  recall that $K_a$ is given by
		\[
		K_a(\omega)=\sup_{n\in \Z} (K(\sigma^n \omega)e^{-a|n|}), \quad \omega \in \Omega.
		\]
	\end{remark}
	
	\subsection{Quenched statistical stability}\label{QSS}
	Let $\mathcal B_w=(\mathcal B_w, \lVert \cdot \rVert_w)$ and $\mathcal B_s=(\mathcal B_s, \lVert \cdot \rVert_s)$ be two Banach spaces such that $\mathcal B_s$ is embedded in $\mathcal B_w$ and that $\lVert \cdot \rVert_w \le \lVert \cdot \rVert_s$ on $\mathcal B_s$. 	
	
	In addition, let $I\subset (-1, 1)$ be an arbitrary interval such that $0\in I$ and assume that for $\omega \in \Omega$ and $\epsilon \in I$, $\mathcal L_{\omega, \epsilon}$ is a bounded operator on both spaces $\mathcal B_w$ and $\mathcal B_s$. We will denote $\mathcal L_{\omega, 0}$ simply by $\mathcal L_\omega$. Finally, we assume that  $\psi\in \mathcal B_s'$ is a nonzero functional, that admits a bounded extension to $\mathcal B_w$ (that we still denote by $\psi$) such that 
	\begin{equation}\label{psi} \L_{\omega,\epsilon}^*\psi=\psi \quad  \text{for $\omega \in \Omega$ and $\epsilon \in I$.} \end{equation}
	
	For $\omega \in \Omega$, $\epsilon \in I$ and $n\in \N$, set
	\begin{equation} 
		\L_{\omega, \epsilon}^n:=\L_{\sigma^{n-1} \omega, \epsilon} \circ \ldots \L_{\sigma \omega, \epsilon} \circ \L_{\omega, \epsilon}.
	\end{equation}
	The following is our quenched statistical stability result. 
	\begin{theorem}\label{SS}
		Let $(\epsilon_k)_{k\in \N}$ be a sequence in $I\setminus \{0\}$ such that $\lim_{k\to \infty} \epsilon_k=0$ and write $\mathcal L_{\omega, k}$ instead of $\mathcal L_{\omega, \epsilon_k}$. We assume that there exists a $\sigma$-invariant set of full measure $\Omega'\subset \Omega$ and $C\in \mathcal K$ such that:
		\begin{itemize}
			\item for $\omega \in \Omega'$ and $k\in \N_0$, there exists $h_{\omega, k}\in \mathcal B_s$ such that $\psi (h_{\omega, k})=1$,  \begin{equation}\label{rd} \L_{\omega, k} h_{\omega, k}=h_{\sigma \omega, k},\end{equation}  and 
			\begin{equation}\label{c1}
				\| h_{\omega, k} \|_s \le C(\omega);
			\end{equation}
			\item for $\omega \in \Omega'$, $h\in \mathcal B_s$ and $n\in \N$,
			\begin{equation}\label{c2}
				\|\L_{\sigma^{-n}\omega}^{n}h\|_w\le C(\omega)\|h \|_w;
			\end{equation}
			\item there exists $\lambda >0$ such that 
			\begin{equation}\label{c3}
				\|\L_{\omega}^{n}h\|_w\leq C(\omega) e^{-\lambda n}\|h\|_w,
			\end{equation}
			for $\omega \in \Omega'$, $n\in \N$ and $h\in V_s$, where
			\[
			V_s:=\{h \in \mathcal B_s: \psi(h)=0\};
			\]
			\item for $\omega \in \Omega'$, $k\in \N$ and $h\in \mathcal B_s$, 
			\begin{equation}\label{c4}
				\|(\L_{\omega,k}-\L_{\omega})h\|_w\leq C(\omega)|\epsilon_k|\|h\|_s. 
			\end{equation}
		\end{itemize}
		Then, there exist $\tilde C\in \mathcal K$ and $r>0$ such that 
		\begin{equation}\label{eq:statstability}
			\|h_{\omega,k}-h_\omega\|_w\le \tilde C(\omega)|\epsilon_k|^r,
		\end{equation}
		for $\mathbb P$-a.e. $\omega \in \Omega$ and $k\in \N$, where $h_\omega:=h_{\omega, 0}$. In particular,
		\[
		\lim_{k\to \infty} h_{\omega, k} =h_\omega \ \text{in $\mathcal B_w$,  for $\mathbb P$-a.e. $\omega \in \Omega$.}
		\]
		
	\end{theorem}
		\begin{remark}
		Our Theorem \ref{SS} can be applied to expanding on average cocycles of piecewise monotone maps as considered by \cite{Buzzi} and thus recover the results of \cite{FGTQ}. 
	\end{remark}
	\begin{proof}
		Take an arbitrary $a$ such that $0<a< \min \{\frac 1 2 , \frac{\lambda}{2} \}$, and let $C_a$ be the tempered random variable given by Proposition~\ref{temp} (corresponding to $C$). For an arbitrary $N\in \N$, by using the property~\eqref{rd}, 
		we have that 
		\begin{equation}\label{a1}
			\begin{split}
				\|h_{\omega,k}-h_{\omega}\|_w&=\|\L_{\sigma^{-N}\omega,k}^{N}h_{\sigma^{-N}\omega,k}-\L_{\sigma^{-N}\omega}^{N}h_{\sigma^{-N}\omega} \|_w\\
				&\leq \|(\L_{\sigma^{-N}\omega,k}^{N}-\L_{\sigma^{-N}\omega}^{N})h_{\sigma^{-N}\omega,k}\|_w+\|\L_{\sigma^{-N}\omega}^{N}(h_{\sigma^{-N}\omega,k}-h_{\sigma^{-N}\omega})\|_w,
			\end{split}
		\end{equation}
		for $\omega \in \Omega'$ and $k\in \N$. Since $\psi(h_{\sigma^{-N}\omega,k})=\psi(h_{\sigma^{-N}\omega})=1$, we have that  $h_{\sigma^{-N}\omega,k}-h_{\sigma^{-N}\omega}\in V_s$, and thus  it follows from~\eqref{c1} and~\eqref{c3} that 
		\begin{equation}\label{a2} \|\L_{\sigma^{-N}\omega}^{N}(h_{\sigma^{-N}\omega,k}-h_{\sigma^{-N}\omega})\|_w\leq 2C(\sigma^{-N}\omega)^2 e^{-\lambda N}, \end{equation}
		for $\omega \in \Omega'$ and $k\in \N$. Moreover, since
		\[\L_{\sigma^{-N}\omega,k}^{N}-\L_{\sigma^{-N}\omega}^{N}=\sum_{j=1}^N\L_{\sigma^{-(N-j)}\omega}^{N-j}(\L_{\sigma^{-N+j-1}\omega,k}-\L_{\sigma^{-N+j-1}\omega})\L^{j-1}_{\sigma^{-N}\omega,k},\]
		by~\eqref{rd}, \eqref{c1}, \eqref{c2} and~\eqref{c4} we have that 
		\begin{align*}
			&\|(\L_{\sigma^{-N}\omega,k}^{N}-\L_{\sigma^{-N}\omega}^{N})h_{\sigma^{-N}\omega,k}\|_w\\ 
			&\leq \sum_{j=1}^N\|\L_{\sigma^{-(N-j)}\omega}^{N-j}\|_{B_w\to B_w}\|(\L_{\sigma^{-N+j-1}\omega,k}-\L_{\sigma^{-N+j-1}\omega})h_{\sigma^{-N+j-1}\omega,k}\|_w\\
			&\leq C(\omega)|\epsilon_k|\ \sum_{j=1}^N C(\sigma^{-N+j-1}\omega)^2 \\
			&\le C(\omega)|\epsilon_k| N\max_{0\le j\le N-1}C(\sigma^{-N+j}\omega)^2.
		\end{align*}
		Hence,  by using~\eqref{sw} (for $C$ and $C_a$ instead of $K$ and $K_a$ respectively), we conclude that 
		\begin{equation}\label{a3}
			\|(\L_{\sigma^{-N}\omega,k}^{N}-\L_{\sigma^{-N}\omega}^{N})h_{\sigma^{-N}\omega,k}\|_w \le C(\omega)C_a(\omega)^2 |\epsilon_k| Ne^{2aN},
		\end{equation}
		for $\mathbb P$-a.e. $\omega \in \Omega$ and $k\in \N$. It follows from~\eqref{sw}, \eqref{a1}, \eqref{a2} and~\eqref{a3} that 
		\[
		\|h_{\omega,k}-h_{\omega}\|_w \le 2C_a(\omega)^2 e^{-(\lambda-2a)N}+C(\omega)C_a(\omega)^2 |\epsilon_k| Ne^{2aN},
		\]
		for $\mathbb P$-a.e. $\omega \in \Omega$ and $k\in \N$. Choosing $N=\lfloor | \log |\epsilon_k| |\rfloor$, yields the desired result. 
	\end{proof}
	
	\begin{remark}
		We stress that the requirement that conditions~\eqref{c1}, \eqref{c2}, \eqref{c3} and~\eqref{c4} hold with the same $C\in \mathcal K$ does not impose any restriction. Indeed, if conditions \eqref{c1}-\eqref{c4} are fulfilled with $C_i\in \mathcal K$, for $i\in \{1, 2, 3, 4\}$ respectively, then we can take $C=\max \{C_1, C_2, C_3, C_4\}$ (see Remark~\ref{TR}).
	\end{remark}

	\subsection{Quenched linear response}
	In order to formulate our abstract quenched linear response result for random dynamics, besides requiring the existence of spaces $\mathcal B_w$ and $\mathcal B_s$ as in Subsection~\ref{QSS}, we also require the existence of a third space $\mathcal B_{ss}=(\mathcal B_{ss}, \lVert \cdot \rVert_{ss})$ that can be embedded in $\mathcal B_s$
	and such that $\lVert \cdot \rVert_s \le \lVert \cdot \rVert_{ss}$ on $\mathcal B_{ss}$.  
	
	As in Subsection~\ref{QSS}, we assume that $\psi$ is a nonzero functional on $\mathcal B_s$, and we shall also assume that it admits a bounded extension to $\mathcal B_w$. We still denote its restriction (resp. extension) to $\mathcal B_{ss}$ (resp. $\mathcal B_w$) by $\psi$. We denote $V_{ss}:=\{h\in \mathcal B_{ss}: \psi(h)=0\}$ and let $V_s$ and $V_w$ denote images of $V_{ss}$ via the natural injections $B_{ss}\hookrightarrow B_s$ and $B_{ss}\hookrightarrow B_w$ respectively. 
	
	Furthermore, we let $(\mathcal L_{\omega,\epsilon})_{\omega \in \Omega,\epsilon\in I}$ be a family such that each
	$\mathcal L_{\omega,\epsilon}$ is a bounded operator on each of those three spaces satisfying~\eqref{psi}, where $I\subset (-1, 1)$ is an interval that contains $0$. We continue to denote $\mathcal L_{\omega, 0}$ by $\mathcal L_\omega$. Moreover, we suppose that $\omega \mapsto \L_{\omega}$ is strongly measurable on $\mathcal B_w$, i.e. for each $h\in \mathcal B_w$ the map $\omega \mapsto \mathcal L_{\omega} h$ is measurable. 
	
	\begin{theorem}\label{LR}
		Let $(\epsilon_k)_{k\in \N}$ be a sequence in $I\setminus \{0\}$ such that $\lim_{k\to \infty} \epsilon_k=0$ and write $\mathcal L_{\omega, k}$ instead of $\mathcal L_{\omega, \epsilon_k}$. We assume that there exists a  $\sigma$-invariant set of full measure $\Omega'\subset \Omega$ and $C\in \mathcal K$ such that:
		\begin{itemize}
			\item for $\omega \in \Omega'$ and $k\in \N_0$, there exists $h_{\omega, k}\in \mathcal B_{ss}$ such that $\L_{\omega, k} h_{\omega, k}=h_{\sigma \omega, k}$, $\psi (h_{\omega, k})=1$ and 
			\begin{equation}\label{c11}
				\| h_{\omega, k} \|_{ss} \le C(\omega).
			\end{equation}
			Moreover,  suppose that $\omega \mapsto h_{\omega, 0}$ is measurable;
			\item for $\omega \in \Omega'$, $h\in \mathcal B_i$, $i\in \{w,s\}$ and $n\in \N$,
			\begin{equation}\label{c22}
				\|\L_{\sigma^{-n}\omega}^{n}h\|_i\le C(\omega)\|h \|_i;
			\end{equation}
			\item  there exists $\lambda >0$ such that for any $\omega \in \Omega'$,  $n\in \N$ and $i\in\{w,s\}$, 
			\begin{equation}\label{c33}
				\| \L_\omega^{n} h\|_{i} \le C(\omega)e^{-\lambda n} \|h \|_{i} \quad \text{for $n\in \N$ and $h \in V_{i}$;}
			\end{equation}
			\item for $\omega \in \Omega'$,  $h\in \mathcal B_s$ and $k\in \N$,
			\begin{equation}\label{c44}
				\| (\L_{\omega, k}-\L_\omega)h \|_w \le C(\omega)|\epsilon_k|\cdot \|h \|_{s}.
			\end{equation}
			Moreover, 
			\begin{equation}\label{c55}
				\| (\L_{\omega,k}-\L_\omega)h \|_s \le C(\omega)|\epsilon_k|\cdot \|h \|_{ss},
			\end{equation}
			for $\omega \in \Omega'$, $h\in \mathcal B_{ss}$ and $k\in \N$;
			\item for $\omega \in \Omega'$, there exists a bounded operator $\hat \L_\omega: B_{ss}\to V_w$ such that 
			\begin{equation}\label{c66}
				\bigg{\|} \frac{1}{\epsilon_k} (\L_{\omega,k}-\L_{\omega})h-\hat \L_\omega h \bigg{\|}_w\le C(\omega)|\epsilon_k|\|h\|_{ss},
			\end{equation}
			for $\omega \in \Omega'$, $h\in \mathcal B_{ss}$ and  $k\in \N$.  Moreover, assume that for each $k\in \N$ and $h\in \mathcal B_{ss}$, $\omega \mapsto \hat \L_\omega h$ is measurable. 
		\end{itemize}
		Then, there exist a measurable $\omega\in\Omega\mapsto\hat h_\omega\in V_w$, $K\in \mathcal K$ and $r>0$ such that 
		\begin{equation}\label{eq:diffequivdens}
			\left\| \frac{1}{\epsilon_k}(h_{\omega,k}-h_{\omega})-\hat h_\omega\right\|_w \le K(\omega)|\epsilon_k|^r,
		\end{equation}
		for $\mathbb P$-a.e. $\omega \in \Omega$ and $k\in \N$, where $h_\omega:=h_{\omega, 0}$. In addition, 
		\begin{equation}\label{eq:linresp}
			\hat h_\omega:=\sum_{n=0}^\infty\L_{\sigma^{-n}\omega}^{n}\hat\L_{\sigma^{-(n+1)}\omega}h_{\sigma^{-(n+1)}\omega}, \quad \text{for $\mathbb P$-a.e. $\omega \in \Omega$.}
		\end{equation}
	\end{theorem}

		\begin{remark}
		Theorem \ref{LR} is a direct generalization to the quenched case of the `quantitative stability for fixed points of Markov operators' result in \cite{GS}: when $\Omega=\{\omega_0\}$ is reduced to a singleton, we recover \cite[Theorem 1]{GS} (see also \cite[Theorem 3]{GG}).
	\end{remark}
	
	\begin{remark}
		The main novelty of Theorem~\ref{LR} when compared with the similar results in the literature (see for example~\cite[Theorem 12]{DS1}) is that we allow for $C$ to depend on $\omega$. Our requirement that $\lambda$ appearing in~\eqref{c33} is independent on $\omega$ does not represent a serious restriction: indeed, in applications $\lambda$ can be expressed in terms of the second Lyapunov exponent associated with the cocycle $(\L_\omega)_{\omega \in \Omega}$ acting on $\mathcal B_s$ and $\mathcal B_w$, which due to our ergodicity assumption for the base space $(\Omega, \mathcal F, \mathds P, \sigma)$ is a deterministic quantity.  Finally, we note that in our applications, the requirement that $C\in \mathcal K$ will again be fulfilled by applying known results from the multiplicative ergodic theory. We refer to Lemma \ref{aux} and Proposition~\ref{prop} for details.
	\end{remark}
	
	\begin{proof}
		We first show that $\hat h_\omega$ given by~\eqref{eq:linresp} is well-defined for $\mathbb P$-a.e. $\omega \in \Omega$. Indeed, observe that~\eqref{c33}, \eqref{c44} and~\eqref{c66} imply that 
		\[
		\begin{split}
			&\bigg{\lVert} \sum_{n=0}^\infty\L_{\sigma^{-n}\omega}^{n}\hat\L_{\sigma^{-(n+1)}\omega}h_{\sigma^{-(n+1)}\omega} \bigg{\rVert}_w  \\
			&\le \sum_{n=0}^\infty C(\sigma^{-n} \omega)e^{-\lambda n} \|\hat\L_{\sigma^{-(n+1)}\omega}h_{\sigma^{-(n+1)}\omega} \|_w \\
			&\le \sum_{n=0}^\infty C(\sigma^{-n} \omega)e^{-\lambda n}\bigg{\lVert}\frac{1}{\epsilon_1} (\L_{\sigma^{-(n+1)} \omega,1}-\L_{ \sigma^{-(n+1)}\omega})h_{\sigma^{-(n+1)}\omega} \bigg \|_w \\
			&\phantom{\le}+\sum_{n=0}^\infty C(\sigma^{-n} \omega)e^{-\lambda n}\bigg \lVert \bigg (\frac{1}{\epsilon_1} (\L_{\sigma^{-(n+1)}\omega, 1}-\L_{\sigma^{-(n+1)}\omega})-\hat \L_{\sigma^{-(n+1)} \omega} \bigg )h_{\sigma^{-(n+1)}\omega} \bigg \|_w \\
			&\le \sum_{n=0}^\infty C(\sigma^{-n} \omega)e^{-\lambda n}C(\sigma^{-(n+1)} \omega) \|h_{\sigma^{-(n+1)}\omega} \|_{ss}\\
			&\phantom{\le}+|\epsilon_1| \sum_{n=0}^\infty C(\sigma^{-n} \omega)e^{-\lambda n}C(\sigma^{-(n+1)} \omega) \|h_{\sigma^{-(n+1)}\omega} \|_{ss} \\
			&\le (1+|\epsilon_1|)\sum_{n=0}^\infty C(\sigma^{-n} \omega)e^{-\lambda n}C(\sigma^{-(n+1)} \omega)^2 \\
			&\le (1+|\epsilon_1|)e^2 C_{\lambda/4} (\omega)^3   \sum_{n=0}^\infty e^{-\frac{\lambda}{4} n} <+\infty,
		\end{split}
		\]
		for $\mathbb P$-a.e. $\omega \in \Omega$, where $C_{\lambda/4}$ is given by Proposition~\ref{temp}. Hence, $\hat h_\omega$ is well-defined for $\mathbb P$-a.e. $\omega \in \Omega$. In addition, due to our measurability assumptions for maps $\omega \mapsto \L_\omega$, $\omega \mapsto \hat \L_\omega$ and 
		$\omega \mapsto h_\omega$, we conclude that $\omega \mapsto \hat h_\omega$ is measurable.
		
		In order to establish the statement of the theorem, we begin by observing that by introducing $   \L_{\sigma^{-1}\omega} h_{\sigma^{-1} \omega,k} $, we have that 
		\begin{equation}\label{i}
			\begin{split}
				h_{\omega,k}-h_\omega & = \L_{\sigma^{-1}\omega,k}h_{\sigma^{-1} \omega,k} - \L_{\sigma^{-1}\omega} h_{\sigma^{-1} \omega} \\
				&	=  \L_{\sigma^{-1}\omega}(h_{\sigma^{-1} \omega,k}-h_{\sigma^{-1} \omega}) +  (\L_{\sigma^{-1}\omega,k}-\L_{\sigma^{-1}\omega})h_{\sigma^{-1}\omega,k},
			\end{split}
		\end{equation}
		for $\omega \in \Omega'$ and $k\in \N$. By iterating~\eqref{i} and noting that $ \L_{\sigma^{-1} \omega} \circ \L_{\sigma^{-2} \omega}\circ \ldots  \circ \L_{\sigma^{-N}\omega} = \L^N_{\sigma^{-N} \omega}$, we have that
		\begin{equation}\label{tel}
			h_{\omega,k}-h_\omega =  \L_{\sigma^{-N}\omega}^{N}(h_{\sigma^{-N} \omega,k}-h_{\sigma^{-N} \omega}) +  \sum_{n=0}^{N-1} \L^n_{\sigma^{-n} \omega} (\L_{\sigma^{-(n+1)}\omega,k}-\L_{\sigma^{-(n+1)}\omega})h_{\sigma^{-(n+1)}\omega,k},
		\end{equation}
		for $\omega \in \Omega'$ and $k, N\in \N$. On the other hand, \eqref{c11} and~\eqref{c33} imply that
		\begin{equation}\label{tel1}
			\|\L_{\sigma^{-N}\omega}^{N}(h_{\sigma^{-N} \omega,k}-h_{\sigma^{-N} \omega})\|_w \le 2C(\sigma^{-N} \omega)^2 e^{-\lambda N} \le 2C_{\lambda/4}(\omega)^2 e^{-\frac{\lambda}{2} N},
		\end{equation}
		for $\mathbb P$-a.e. $\omega \in \Omega$ and  $k, N\in \N$.
		Dividing~\eqref{tel} by $\epsilon_k$, letting $N\to \infty$ and using~\eqref{tel1}, yields  that
		\begin{equation}\label{dec} 
			\frac{1}{\epsilon_k} (h_{\omega,k}-h_\omega) =\sum_{n=0}^\infty \L_{\sigma^{-n} \omega}^{n} \bigg ( \frac{1}{\epsilon_k}(\L_{\sigma^{-(n+1)}\omega,k}-\L_{\sigma^{-(n+1)}\omega})h_{\sigma^{-(n+1)}\omega,k} \bigg ),
		\end{equation}
		for $\mathbb P$-a.e. $\omega \in \Omega$ and $k\in \N$.
		Thus, it follows from~\eqref{eq:linresp} and~\eqref{dec} that
		\[
		\frac{1}{\epsilon_k} (h_{\omega,k}-h_\omega)-\hat h_\omega=(I)+(II),
		\]
		where
		\[
		(I):=\sum_{n=0}^\infty \L_{\sigma^{-n} \omega}^{n} \bigg ( \frac{1}{\epsilon_k}(\L_{\sigma^{-(n+1)}\omega,k }-\L_{\sigma^{-(n+1)}\omega})-\hat\L_{\sigma^{-(n+1)}\omega}\bigg )h_{\sigma^{-(n+1)}\omega}
		\]
		and 
		\[
		(II):=\sum_{n=0}^\infty \L_{\sigma^{-n} \omega}^{n} \bigg ( \frac{1}{\epsilon_k}(\L_{\sigma^{-(n+1)}\omega,k }-\L_{\sigma^{-(n+1)}\omega} )(h_{\sigma^{-(n+1)}\omega,k}-h_{\sigma^{-(n+1)}\omega}) \bigg ) .
		\]
		Now, applying Theorem \ref{SS} for  the pair of  spaces $(\mathcal B_{ss}, \mathcal B_s)$ and using~\eqref{c33} and~\eqref{c44}, we observe that
		\[
		\begin{split}
			& \bigg{\|}  \L_{\sigma^{-n} \omega}^{n} \bigg ( \frac{1}{\epsilon_k}(\L_{\sigma^{-(n+1)}\omega,k }-\L_{\sigma^{-(n+1)}\omega} )(h_{\sigma^{-(n+1)}\omega,k}-h_{\sigma^{-(n+1)}\omega}) \bigg )  \bigg{\|}_w \\
			&\le  C(\sigma^{-n}\omega)e^{-\lambda n} \bigg{\|}  \frac{1}{\epsilon_k}(\L_{\sigma^{-(n+1)}\omega,k }-\L_{\sigma^{-(n+1)}\omega} )(h_{\sigma^{-(n+1)}\omega,k}-h_{\sigma^{-(n+1)}\omega}) \bigg{\|}_w \\
			&\le  C(\sigma^{-n}\omega)e^{-\lambda n} C(\sigma^{-(n+1)}\omega)\|h_{\sigma^{-(n+1)}\omega,k}-h_{\sigma^{-(n+1)}\omega}\|_s \\
			&\le C(\sigma^{-n}\omega)e^{-\lambda n} C(\sigma^{-(n+1)}\omega)\tilde C(\sigma^{-(n+1)}\omega)|\epsilon_k|^r,
		\end{split}
		\]
		for $\mathbb P$-a.e. $\omega \in \Omega$, $n\in \N$ and $k\in \N$, where $\tilde C$ and $r>0$ are given by Theorem~\ref{SS}. By applying Proposition~\ref{temp}, we find that 
		\[
		\begin{split}
			& \bigg{\|}  \L_{\sigma^{-n} \omega}^{n} \bigg ( \frac{1}{\epsilon_k}(\L_{\sigma^{-(n+1)}\omega,k }-\L_{\sigma^{-(n+1)}\omega} )(h_{\sigma^{-(n+1)}\omega,k}-h_{\sigma^{-(n+1)}\omega}) \bigg )  \bigg{\|}_w \\
			&\le e^2 C_{\lambda/4}(\omega)^2 \tilde C_{\lambda/4}(\omega)e^{-\frac{\lambda}{4} n}|\epsilon_k|^r,
		\end{split}
		\]
		and therefore we conclude that  there exists $D_1 \in \mathcal K$ such that 
		\begin{equation}\label{EQ1}
			\| (II)\|_w \le D_1(\omega)|\epsilon_k|^r, \quad \text{for $\mathbb P$-a.e. $\omega \in \Omega$ and $k\in \N$.}
		\end{equation}
		
		On the other hand, by~\eqref{c11}, \eqref{c33}, \eqref{c66} and Proposition~\ref{temp} we have that 
		\[
		\begin{split}
			&\bigg \|  \L_{\sigma^{-n} \omega}^{n} \bigg ( \frac{1}{\epsilon_k}(\L_{\sigma^{-(n+1)}\omega,k }-\L_{\sigma^{-(n+1)}\omega})-\hat\L_{\sigma^{-(n+1)}\omega}\bigg )h_{\sigma^{-(n+1)}\omega} \bigg \|_w \\
			&\le C(\sigma^{-n}\omega)e^{-\lambda n}\bigg \|   \bigg ( \frac{1}{\epsilon_k}(\L_{\sigma^{-(n+1)}\omega,k }-\L_{\sigma^{-(n+1)}\omega})-\hat\L_{\sigma^{-(n+1)}\omega}\bigg )h_{\sigma^{-(n+1)}\omega} \bigg \|_w \\
			&\le |\epsilon_k| C(\sigma^{-n}\omega)e^{-\lambda n}C(\sigma^{-(n+1)}\omega) \|h_{\sigma^{-(n+1)}\omega}  \|_{ss} \\
			&\le  |\epsilon_k| C(\sigma^{-n}\omega)e^{-\lambda n}C(\sigma^{-(n+1)}\omega)^2 \\
			&\le e^2 |\epsilon_k| C_{\lambda/4}(\omega)^3 e^{-\frac{\lambda}{4} n},
		\end{split}
		\]
		for $\mathbb P$-a.e. $\omega \in \Omega$ and $n, k\in \N$. Thus, there exists  $D_2\in \mathcal K$ such that 
		\begin{equation}\label{EQ2}
			\| (I)\|_w \le D_2(\omega)|\epsilon_k|, \quad \text{for $\mathbb P$-a.e. $\omega \in \Omega$ and $k\in \N$.}
		\end{equation}
		Finally, we observe that it follows from the estimates~\eqref{EQ1} and~\eqref{EQ2} that~\eqref{eq:diffequivdens} holds with $K=D_1+D_2\in \mathcal K$.
		The proof of the theorem is completed.
	\end{proof}

	\section{Parameterized smooth expanding on average cocycles}\label{PSC}
	In this section, we introduce and study  parameterized smooth expanding on average cocycles.
	More precisely, we establish several auxiliary results  in which we essentially verify that  assumptions of Theorem~\ref{LR} are fulfilled in our setting. 
	
	Throughout this section, $\mathbb S^1$ will denote the unit circle endowed with  the Lebesgue measure $m$. 
	
	We begin by introducing the notion of a parameterized smooth expanding on average cocycle. 
	\begin{definition}\label{EC}
		Let $(\Omega, \mathcal F, \mathbb P)$ be a probability space such that $\Omega$ is  a Borel subset of a separable complete metric space. Furthermore, let $\sigma \colon \Omega \to \Omega$ be a homeomorphism that preserves $\mathbb P$ and that $\mathbb P$ is ergodic. Moreover, take an interval
		$I\subset (-1,1)$ that  contains $0$ and let $r\in \N$, $r\ge 4$. We say that  a measurable map $\mathbf T \colon \Omega \to C^r(I\times \mathbb S^1, \mathbb S^1)$ is a \emph{parameterized smooth expanding on average cocycle} if the following holds:
		\begin{itemize}
			\item there exists a log-integrable random variable $K\colon \Omega \to (0, \infty)$ such that  for $\epsilon \in I$ and $\mathbb P$-a.e. $\omega \in \Omega$, 
			\begin{equation}\label{n1}
				\|T_{\omega, \epsilon}\|_{C^r}\le K(\omega)
			\end{equation}
			and, for $i\in\{1,2\}$,
			\begin{equation}\label{n2}
				\| \partial_\epsilon^i T_{\omega, \epsilon}\|_{C^{r-i}}\le K(\omega),
			\end{equation}
			where $T_{\omega, \epsilon}:=\mathbf T(\omega)(\epsilon, \cdot)\in C^r(\mathbb S^1, \mathbb S^1)$;
			\item for each $\epsilon \in I$, $\omega \mapsto \lambda_{\omega, \epsilon}:=\min |T_{\omega, \epsilon}'|$ and $\omega \mapsto \int_{\mathbb S^1} \frac{|T_{\omega, \epsilon}''|}{(T_{\omega, \epsilon}')^2}\, dm$ are measurable;
			\item there exists a random variable $\underline{\lambda} \colon \Omega \to (0, +\infty)$ such that for any $\epsilon \in I$, $\lambda_{\omega, \epsilon} \ge \underline{\lambda}(\omega)$ for $\mathbb P$-a.e. $\omega \in \Omega$. Moreover, 
			\begin{equation}\label{lambda}
				\int_{\Omega} \log \underline{\lambda}(\omega)\, d\mathbb P(\omega)>0.
			\end{equation}
		\end{itemize}
	\end{definition}
	\begin{remark}
		It follows from~\eqref{lambda} that $\underline{\lambda} \in \mathcal K$ (see Remark~\ref{TR}).
	\end{remark}
	\begin{remark}
		The class of systems introduced in Definition~\ref{EC} admits a natural higher dimensional counterpart: on the $d$-dimensional torus $\mathbb T^d$, endowed with its Haar-Lebesgue measure $m$, one may consider a measurable mapping $\mathbf T:\Omega\to C^r(I\times \mathbb T^d,\mathbb T^d)$ for some $r>4$. Setting $T_{\omega,\epsilon}:=\mathbf{T}(\omega)(\epsilon,\cdot)$, we assume that:
		\begin{itemize}
		 \item \eqref{n1} and \eqref{n2} hold. 
		 \item for each $\epsilon\in I$, $\lambda_{\omega,\epsilon}:=\min_{x\in\mathbb T^d}\min_{\underset{\|v\|=1}{v\in\mathbb R^d}}\|DT_{\omega,\epsilon}(x)\cdot v\|$ is measurable. 
		 \\Furthermore, there exists a random variable  $\underline\lambda:\Omega\to(0,+\infty)$ such that for any $\epsilon \in I$, $\lambda_{\omega, \epsilon}\ge\underline\lambda (\omega)$ for $\mathbb P$-a.e $\omega \in \Omega$. In addition, we suppose that  \eqref{lambda} holds;
		 \item setting $g_{\omega,\epsilon}:=\det(DT_{\omega,\epsilon}(\cdot))$, we assume that $\omega \mapsto \int_{\mathbb T^d}\|Dg_{\omega,\epsilon}(x)\|dm(x)$ is $\log$-integrable (with respect to $\mathbb P$), for each $\epsilon \in I$.
		\end{itemize}
	All the results of the present section apply also in this setting, modulo some obvious changes.
	\end{remark}
	\begin{example}\label{EXAMP}
		Let $(\Omega, \mathcal F, \mathbb P)$, $\sigma \colon \Omega \to \Omega$ and $r\in \N$ be as in Definition~\ref{EC}.
		Choose a $\log$-integrable  random variable $\beta \colon \Omega \to \N$ such that  $\mathbb P(\{\beta >1\})>0$. Hence, 
		\[
		\int_\Omega \log \beta (\omega)\, d\mathbb P(\omega)>0.
		\]
		For $\omega \in \Omega$, we define $T_\omega \colon \mathbb S^1 \to \mathbb S^1$ by
		\[
		T_\omega (x)=\beta (\omega)x \,  (mod \ 1), \quad x\in \mathbb S^1.
		\]
		Furthermore, let $D\colon \Omega \to C^r(\mathbb S^1, \mathbb S^1)$ be a measurable map such that $\omega \mapsto \|D_\omega\|_{C^r}$ is  $\log$-integrable  and $D_\omega'(x)\ge 0$ for $x\in \mathbb S^1$ and $\omega \in \Omega$. We define $\mathbf T \colon \Omega \to C^r((-1, 1)\times \mathbb S^1, \mathbb S^1)$ by
		\[
		\mathbf T(\omega)(\epsilon, \cdot)=(\Id+\epsilon^2D_\omega) \circ T_\omega.
		\]
		It is straightforward to verify that $\mathbf T$ is a parameterized smooth expanding on average cocycle. In particular, using the same notation as in Definition~\ref{EC}, we have $\underline{\lambda}=\beta$.
	\end{example}
	\begin{example}
		Consider, for $m\in\mathbb N$ and $r>4$, a family $\{T_0,\dots,T_m \}$ of local $C^r$ diffeomorphisms of $\mathbb S^1$, and a family  $\{d_0,\dots,d_m\}$ of  $C^r$ mappings $d_i:\mathbb S^1\to\mathbb R$. Setting $\lambda_i:=\min_{x\in\mathbb S^1}|T'_i(x)|>0$, we assume that $\sum_{i=0}^{m-1}\log\lambda_i>0$.
		\\Consider now the full-shift on the alphabet $\{0,\dots,m\}$, i.e. set $\Omega:=\{0,\dots,m\}^{\mathbb Z}$, and for $\omega=(\omega_n)_{n\in \mathbb Z}$, $(\sigma\omega)_n=\omega_{n+1}$, endowed with  a Markov measure $\mathbb P$ associated to a probability vector $(p_0,\dots,p_m)$.
		Let $\epsilon_0>0$ be small enough, and set $I=(-\epsilon_0,\epsilon_0)$. We then define a measurable mapping $\mathbf T:\Omega\to C^r(I\times\mathbb S^1,\mathbb S^1)$ by
		\begin{equation*}
			\mathbf T(\omega)(\epsilon,\cdot):= T_i+\epsilon d_i~\text{if}~\omega_0=i.
		\end{equation*}
	This data defines a cocycle over $\sigma$, which is simply the i.i.d composition of maps $T_{\omega,\epsilon}=\mathbf T(\omega)(\epsilon,\cdot)$. It is easy to see that it satisfies the requirements of Definition \ref{EC}: in particular, \eqref{n1}, \eqref{n2} are satisfied with constant $K$, and the expansion on average condition is satisfied with $\underline\lambda(\omega):=\lambda_i-\epsilon_0\min |d'_i|~\text{if}~\omega_0=i$; \eqref{lambda} holds up to shrinking $\epsilon_0$.
	\end{example}	
	
	The following auxiliary result shows that parameterized smooth expanding on average cocycles exhibit the so-called random covering condition (of uniform type with respect to the parameter $\epsilon$). For the proof, we refer to~\cite[Example 6]{DHS}.
	\begin{proposition}\label{covering}
		Let $\mathbf T\colon \Omega \to C^r(I\times \mathbb S^1, \mathbb S^1)$ be a parameterized  smooth expanding on average cocycle. Then, for any interval $J\subset \mathbb S^1$ and $\mathbb P$-a.e. $\omega \in \Omega$, there exists $n_c=n_c(\omega, J)\in \N$ such that 
		\[
		T_{\omega, \epsilon}^n (J)=\mathbb S^1, \quad \text{for $n\ge n_c$ and $\epsilon \in I$.}
		\]
	\end{proposition}

	The following is the main result of this section: a linear response result for parameterized smooth expanding on average cocycles.
	\begin{theorem}\label{LR1}
		Let $\mathbf T\colon \Omega \to C^r(I\times \mathbb S^1, \mathbb S^1)$ be an arbitrary parameterized smooth expanding on average cocycle. For $\epsilon \in I$, let $(h_{\omega, \epsilon})_{\omega \in \Omega}$ be the family given by Lemma~\ref{1158}. 
		Then, there exists a measurable family $(\hat h_\omega)_{\omega \in \Omega} \subset W^{1,1}$, $r>0$ and $K\in \mathcal K$ such that
		\[
		\left\| \frac{1}{\epsilon_k}(h_{\omega,k}-h_{\omega})-\hat h_\omega\right\|_{W^{1,1}} \le K(\omega)|\epsilon_k|^r,
		\]
		for $\mathbb P$-a.e. $\omega \in \Omega$ and $k\in \N$, where $h_{\omega, k}:=h_{\omega,\epsilon_k}$, $h_\omega:=h_{\omega, 0}$ and $(\epsilon_k)_{k\in \N}$ is a sequence in $I\setminus \{0\}$ such that $\epsilon_k \to 0$. Moreover, $\hat h_\omega$ is given by~\eqref{eq:linresp}, where $\L_\omega$ is the transfer operator associated to $\mathbf T(\omega)(0, \cdot)$ and $\hat \L_\omega$ is given by~\eqref{hatL}.
	\end{theorem}
	The proof of Theorem~\ref{LR1} is given in \S.\ref{sec:LRproof}. 
	\subsection{Lasota-Yorke inequalities}
	
	By $W^{\ell, 1}:=W^{\ell, 1}(\mathbb S^1)$ we will denote the Sobolev space consisting of all $f\in L^1(\mathbb S^1)$ with the property that its derivatives $f^{(j)}$ exist in the weak sense for $j\le \ell$ and $\| f^{(j)}\|_{L^1}<+\infty$. Then, $W^{\ell, 1}(\mathbb S^1)$ is a Banach space with respect to the norm
	\[
	\|f\|_{W^{\ell, 1}}:=\sum_{0\le j\le \ell} \|f^{(j)}\|_{L^1}.
	\]
	We start by recalling the following well-known lemma (see for example~\cite[Lemma 5.9]{HC}).
	\begin{lemma}\label{lemma:Crim}
		Let $T\in C^r(\mathbb S^1, \mathbb S^1)$ and let $\ell\in\N$ be such that $\ell+1\le r$. Then, for any $f\in W^{\ell, 1}$ we have that 
		\[
		\L_T(f)^{(\ell)}=\L_T\left(\frac{1}{(T')^{2\ell}}\sum_{j=0}^\ell G_{\ell,j}(T',\dots,T^{(\ell+1)})f^{(j)}\right),
		\]
		where $G_{\ell,j}(\cdot,\dots,\cdot):\R^{\ell+1} \to\R$ are polynomials, $\L_T$ denotes the transfer operator associated to $T$  and $\L_T(f)^{(\ell)}$ is the $\ell$-derivative of $\L_T(f)$.
	\end{lemma}
	
	\begin{remark} \label{rem:primedifferentiation}
		
		We note that polynomials  $G_{\ell, j}$ can be constructed explicitly via simple recursive relations. Indeed, setting $G_{0,0}:=1$, we have that
		\begin{align*}
			&G_{\ell+1,0}=(1-2\ell)x_2G_{\ell,0}+x_1G_{\ell,0}'\\
			&G_{\ell+1,k}=(1-2\ell)x_2 G_{\ell,k}+x_1(G_{\ell, k}'+G_{\ell, k-1}),\quad\text{for}~k\in\{1,\dots,\ell\}\\
			&G_{\ell+1,\ell+1}=x_1^{\ell+1}.
		\end{align*}
		Here, we define $P'$ of an arbitrary polynomial $P$ by setting $(1)'=0$, $(x_j)'=x_{j+1}$ and by requiring that $(P+Q)'=P'+Q'$ and $(PQ)'=P'Q+PQ'$ for arbitrary polynomials $P$ and $Q$. For example, $(-x_2^2+x_1x_3)'=-2x_2x_3+x_1x_4+x_2x_3=-x_2x_3+x_1x_4$.
		
		It is easy to verify that 
		\begin{equation}\label{der}
			(G_{\ell,j}(T',\dots,T^{(\ell +1)}))'=G_{\ell,j}'(T',\dots, T^{(\ell+2)}),
		\end{equation}
		where $(G_{\ell,j}(T',\dots,T^{(\ell +1)}))'$ denotes the usual derivative of the map $G_{\ell,j}(T',\dots,T^{(\ell +1)})$.
	\end{remark}
	
	
	\begin{lemma}\label{lemma:weakLY}
		Let $\mathbf T \colon \Omega \to C^r(I\times \mathbb S^1, \mathbb S^1)$ be a parameterized smooth expanding on average cocycle. 
		For any $\ell \in \N$ such that $\ell+1\le r$, there exist log-integrable random variables $B_\ell, C_\ell \colon \Omega \to (0, \infty)$ such that, for any $\epsilon \in I$, $\mathbb P$-a.e $\omega\in\Omega$\footnote{Here and throughout the  paper this means that for each $\epsilon \in I$, there exists a full measure set $\Omega_\epsilon \subset \Omega$} and $f\in W^{\ell,1}$,
		\begin{equation}\label{eq:LY}
			\|\L_{\omega, \epsilon}f\|_{W^{\ell,1}}\le C_\ell(\omega)\|f\|_{W^{\ell,1}}
		\end{equation} 
		and
		\begin{equation}\label{eq:LYS}
			\|\L_{\omega, \epsilon }f\|_{W^{\ell,1}}\le \frac{1}{\underline{\lambda}(\omega)^\ell}\|f\|_{W^{\ell,1}}+B_\ell(\omega)\|f\|_{W^{\ell-1,1}},
		\end{equation}
		where $\L_{\omega, \epsilon}$ denotes the transfer operator of $T_{\omega, \epsilon}:=\mathbf T(\omega)(\epsilon, \cdot)$, $\underline{\lambda}(\omega)$ is as in Definition~\ref{EC} and $\| \cdot \|_{W^{0,1}}:=\| \cdot \|_{L^1}$.
	\end{lemma}
	
	\begin{proof}
		Starting from Lemma \ref{lemma:Crim},  applying~\eqref{n1} and using the same notation as in Definition~\ref{EC},  we have that 
		\begin{align*}
			\int_{\mathbb S^1} \left|(\L_{\omega, \epsilon}f)^{(\ell)}\right|\,dm&\le \int_{\mathbb S^1} \left|\frac{1}{(T_{\omega,\epsilon}')^{2\ell}}\sum_{j=0}^\ell G_{\ell,j}(T_{\omega, \epsilon}', \ldots, T_{\omega, \epsilon}^{(\ell+1)}) f^{(j)}\right|\, dm\\
			&\le \lambda_{\omega, \epsilon}^{-2\ell}\max_{0\le j\le \ell}|G_{\ell,j}(T_{\omega, \epsilon}', \ldots, T_{\omega, \epsilon}^{(\ell+1)})|_{\infty}\sum_{j=0}^\ell\int_{\mathbb S^1}|f^{(j)}|\,dm\\
			&\le \underline{\lambda} (\omega)^{-2\ell}\max_{0\le j\le\ell} \tilde G_{\ell, j} (K(\omega),\dots,K(\omega)) \cdot \|f\|_{W^{\ell,1}},
		\end{align*}
		for any $\epsilon \in I$, $\mathbb P$-a.e. $\omega \in \Omega$ and $f\in W^{\ell, 1}$, for some polynomials  $\tilde G_{\ell, j}$ satisfying $\tilde G_{\ell, j} (x_1, \ldots, x_{\ell+1})\ge 0$ for $x_i\ge 0$, $1\le i \le \ell+1$.
		We conclude that~\eqref{eq:LY} holds with
		\[
		C_{\ell}(\omega):=\sum_{i=0}^\ell \underline{\lambda} (\omega)^{-2i}\max_{0\le j\le i} \tilde G_{i, j} (K(\omega),\dots,K(\omega)).
		\]
		Remarking that $\omega \mapsto \max_{0\le j\le i} \tilde G_{i,j}(K(\omega),\dots,K(\omega))$ is log-integrable and recalling that the same holds for $\underline{\lambda}$ , we have that $C_{\ell}$ is log-integrable. 
		
		Furthermore, we have  (recall that $G_{\ell, \ell}=x_1^{\ell}$) that 
		\begin{align*}
			&\int_{\mathbb S^1}\left|(\L_{\omega, \epsilon}f)^{(\ell)}\right|\,dm \le \int_{\mathbb S^1} \left|\frac{1}{(T_{\omega,\epsilon}')^{2\ell}}\sum_{j=0}^\ell G_{\ell,j}(T_{\omega, \epsilon}', \ldots, T_{\omega, \epsilon}^{(\ell+1)})f^{(j)}\right|\, dm\\
			&\le \frac{1}{\lambda_{\omega, \epsilon}^\ell}\int_{\mathbb S^1}|f^{(\ell)}|\,dm+\sum_{j=0}^{\ell-1}\int_{\mathbb S^1} \left|\frac{1}{(T_{\omega,\epsilon}')^{2\ell}} G_{\ell,j}(T_{\omega, \epsilon}', \ldots, T_{\omega, \epsilon}^{(\ell+1)})  f^{(j)}\right|\,dm\\
			&\le \underline{\lambda}(\omega)^{-\ell} \int_{\mathbb S^1}|f^{(\ell)}|\, dm+ \underline{\lambda}(\omega)^{-2\ell}\max_{0\le j\le\ell-1}\tilde G_{\ell,j}(K(\omega),\dots,K(\omega))\sum_{j=0}^{\ell-1}\int_{\mathbb S^1}|f^{(j)}|\,dm,
		\end{align*}
		for  $\epsilon \in I$, $\mathbb P$-a.e. $\omega \in \Omega$ and $f\in W^{\ell, 1}$, which implies that~\eqref{eq:LYS} holds with 
		\[
		B_\ell(\omega):=\underline{\lambda} (\omega)^{-2\ell}\max_{0\le j\le\ell-1} \tilde G_{\ell, j} (K(\omega),\dots,K(\omega))+C_{\ell-1}(\omega).
		\]
		The proof of the lemma is completed. 
	\end{proof}

	\subsection{Lyapunov exponents and Oseledets splittings}
	Throughout this subsection, we fix an arbitrary parameterized smooth expanding on average cocycle $\mathbf T\colon \Omega \to C^r(I\times \mathbb S^1, \mathbb S^1)$. We continue to denote by $\L_{\omega, \epsilon}$ the transfer operator associated to $T_{\omega, \epsilon}:=\mathbf T(\omega)(\epsilon, \cdot)$.

	
	\begin{proposition}\label{1125}
		For each $\epsilon \in I$, the cocycle $(\L_{\omega, \epsilon})_{\omega \in \Omega}$ is strongly measurable and quasi-compact on $W^{\ell, 1}$ for $1\le \ell \le r-1$.
	\end{proposition}
	
	\begin{proof}
		Take an arbitrary $1\le \ell \le r-1$. The strong measurability of $(\L_{\omega, \epsilon})_{\omega \in \Omega}$, for $\epsilon \in I$ follows from \cite[Prop 4.11]{bomfim2016} (by arguing as in~\cite[Proof of Proposition 5.2]{HC}) and Definition~\ref{EC}.

		Observe that~\eqref{eq:LY} implies that the cocycle $(\L_{\omega, \epsilon})_{\omega \in \Omega}$ satisfies integrability assumptions~\eqref{int}. The quasi-compactness now follows directly from~\eqref{lambda}, \eqref{eq:LYS} and Lemma~\ref{QC}.
	\end{proof}
	
	\begin{remark}
		Observe that Proposition~\ref{1125} enables us to apply Theorem~\ref{MET} for the cocycle $(\L_{\omega, \epsilon})_{\omega \in \Omega}$ on $W^{\ell, 1}$, for every $1\le \ell \le r-1$ and $\epsilon \in I$. Since $(\L_{\omega, \epsilon})_{\omega \in \Omega}$ is a cocycle of transfer operators, we have that its largest Lyapunov exponent is zero on each $W^{\ell, 1}$.
	\end{remark}
	\begin{remark}\label{W11}
		Setting $\var(f)=\|f'\|_{L^1}$, we note that the $BV$ space introduced in~\cite[p.10]{DS} coincides with $W^{1,1}$. Thus, 
		it follows  easily from Definition~\ref{EC} and Proposition~\ref{covering} that for each $\epsilon \in I$, the cocycle $(\L_{\omega, \epsilon})_{\omega \in \Omega}$ on $W^{1,1}$ satisfies all requirements of~\cite[Definition 13]{DS} except $\mathbb P$-continuity (see~\cite[Definition 5]{DS}). Nevertheless, all the results from~\cite[Section 2]{DS} are applicable to the cocycle $(\L_{\omega, \epsilon})_{\omega \in \Omega}$ on $W^{1,1}$, for each $\epsilon \in I$. Indeed, the $\mathbb P$-continuity assumption
		imposed in~\cite{DS} ensured that in the context of~\cite[Definition 13]{DS}, the version of the multiplicative ergodic theorem (dealing with cocycles acting on non-separable Banach spaces)  as stated in~\cite[Theorem 9]{DS} can be applied. In the present context, one can easily check that the arguments in~\cite[Section 2]{DS} can be repeated verbatim  by simply applying Theorem~\ref{MET} instead of~\cite[Theorem 9]{DS} when needed. 
	\end{remark}
	
	Taking into account Remark~\ref{W11}, the following result is a direct consequence of~\cite[Theorem 20]{DS} and 
	~\cite[Proposition 24]{DS}.
	\begin{lemma}\label{1158}
		For $\epsilon \in I$, there exists a unique measurable family $(h_{\omega, \epsilon})_{\omega \in \Omega} \subset W^{1,1}$ such that $h_{\omega, \epsilon}\ge 0$, $\int_{\mathbb S^1}h_{\omega, \epsilon}\, dm=1$ and 
		\[
		\L_{\omega, \epsilon}h_{\omega, \epsilon}=h_{\sigma \omega, \epsilon}, \quad \text{for $\mathbb P$-a.e. $\omega \in \Omega$.}
		\]
		Moreover, there is
		$\rho(\epsilon) \in (0, 1)$ such that 
		\[
		\lim_{n\to \infty} \frac 1 n \log  \|\L_{\omega, \epsilon}^n f\|_{W^{1,1}}\le \log (\rho(\epsilon)),
		\]
		for $\mathbb P$-a.e. $\omega \in \Omega$ and $f\in W^{1,1}$, $\int_{\mathbb S^1} f\, dm=0$.
	\end{lemma}

	Our next goal is to extend the second part of Lemma~\ref{1158} to any $W^{\ell, 1}$, $1\le \ell \le r-1$.
	
	\begin{lemma}\label{aux}
		For $\epsilon \in I$ and $1\le \ell \le r-1$, we have that
		\begin{equation}\label{limit}
			\lim_{n\to \infty} \frac 1 n\log  \|\L_{\omega, \epsilon}^n f\|_{W^{\ell,1}}\le \log (\rho(\epsilon)),
		\end{equation}
		for $\mathbb P$-a.e. $\omega \in \Omega$ and $f\in W^{\ell,1}$, $\int_{\mathbb S^1} f\, dm=0$, where $\rho(\epsilon)$ is given by Lemma~\ref{1158}.
	\end{lemma}
	
	\begin{proof}
		Let us fix an arbitrary $\epsilon \in I$.
		We note that the existence of the limit in~\eqref{limit} follows from Theorem~\ref{MET}. Moreover, without any loss of generality, we may assume that 
		\begin{equation}\label{neg}
			-\log \rho(\epsilon)-\int_\Omega \log \underline{\lambda} (\omega)\, d\mathbb P(\omega)<0.
		\end{equation}
		We proceed by induction on $\ell$. For $\ell=1$, the desired conclusion follows from Lemma~\ref{1158}. Assume now that the conclusion holds for some $2\le \ell \le r-2$ and take $f\in W^{\ell+1, 1}$, $\int_{\mathbb S^1}f\, dm=0$. By our induction hypothesis,
		\begin{equation}\label{2056}
			\lim_{n\to \infty} \frac 1 n \log  \|\L_{\omega, \epsilon}^n f\|_{W^{\ell,1}}\le \log (\rho(\epsilon)), \quad \text{for $\mathbb P$-a.e. $\omega \in \Omega$.}
		\end{equation}
		We consider the cocycle $(\bar \L_{\omega, \epsilon})_{\omega \in \Omega}$ over $\sigma$, where $\bar \L_{\omega, \epsilon}=\frac{1}{\rho(\epsilon)}\L_{\omega, \epsilon}$. Observe that~\eqref{2056} implies that 
		\begin{equation}\label{2057}
			\lim_{n\to \infty} \frac 1 n \log  \|\bar \L_{\omega, \epsilon}^n f\|_{W^{\ell,1}}\le 0, \quad \text{for $\mathbb P$-a.e. $\omega \in \Omega$.}
		\end{equation}
		Moreover, \eqref{eq:LY} and~\eqref{eq:LYS} give that for $\mathbb P$-a.e. $\omega \in \Omega$ and $h\in W^{\ell+1, 1}$,
		\begin{equation}\label{eq:LYb}
			\|\bar \L_{\omega, \epsilon}h\|_{W^{\ell+1,1}}\le \frac{1}{\rho(\epsilon)} C_{\ell+1}(\omega)\|f\|_{W^{\ell+1,1}}
		\end{equation} 
		and
		\begin{equation}\label{eq:LYSb}
			\|\bar \L_{\omega, \epsilon }h\|_{W^{\ell+1,1}}\le \frac{1}{\rho(\epsilon) \underline{\lambda}(\omega)^{\ell+1}}\|h\|_{W^{\ell+1,1}}+\frac{1}{\rho(\epsilon)}B_{\ell+1}(\omega)\|h\|_{W^{\ell,1}}.
		\end{equation}
		Furthermore, \eqref{neg} implies that 
		\begin{equation}\label{neg2}
			\int_\Omega \log \bigg (\frac{1}{\rho(\epsilon) \underline{\lambda}(\omega)^{\ell+1}} \bigg )\, d\mathbb P(\omega)<0.
		\end{equation}
		Since the inclusion $W^{\ell+1,1}\hookrightarrow W^{\ell,1}$ is compact, it follows from~\eqref{2057}, \eqref{eq:LYb}, \eqref{eq:LYSb}, \eqref{neg2} and~\cite[Lemma C.5]{GTQ} that 
		\[
		\lim_{n\to \infty} \frac 1 n \log \|\bar \L_{\omega, \epsilon}^n f\|_{W^{\ell+1,1}}\le 0, \quad \text{for $\mathbb P$-a.e. $\omega \in \Omega$.}
		\]
		Hence, 
		\[
		\lim_{n\to \infty} \frac 1 n \log \| \L_{\omega, \epsilon}^n f\|_{W^{\ell+1,1}}\le \log \rho(\epsilon), \quad \text{for $\mathbb P$-a.e. $\omega \in \Omega$.}
		\]
		The proof of the proposition is completed. 
	\end{proof}
	
	Let $Y_1^{\ell, \epsilon}(\omega)\subset W^{\ell, 1}$, $\omega \in \Omega$ denote the Oseledets subspace corresponding to the largest Lyapunov exponent (which is zero) of the cocycle $(\L_{\omega, \epsilon})_{\omega \in \Omega}$ on $W^{\ell, 1}$. 
	\begin{proposition}
		We have that $\dim Y_1^{\ell, \epsilon}(\omega)=1$ and $Y_1^{\ell, \epsilon}(\omega)$ is spanned by $h_{\omega, \epsilon}$, for $\epsilon \in I$, $\mathbb P$-a.e. $\omega \in \Omega$ and $1\le \ell \le r-1$.
	\end{proposition}
	
	\begin{proof}
		Take an arbitrary $\epsilon \in I$.
		For $\ell=1$, the desired conclusions follow directly from Lemma~\ref{1158}. Assume now that $\ell >1$. By arguing as in the proof of~\cite[Lemma 2.10]{D}, we find that there exists $h_{\omega, \epsilon}^\ell \in Y_1^{\ell, \epsilon}(\omega)\subset W^{\ell, 1}$ satisfying $\int_{\mathbb S^1}h_{\omega, \epsilon}^\ell \, dm=1$, $h_{\omega, \epsilon}^\ell \ge 0$ and such that $\L_{\omega, \epsilon} h_{\omega, \epsilon}^{\ell}=
		h_{\sigma \omega, \epsilon}^\ell$ for $\mathbb P$-a.e. $\omega \in \Omega$. Moreover, $\omega \mapsto h_{\omega, \epsilon}^\ell$ is measurable. 
		Since $W^{\ell, 1} \subset W^{1,1}$, it follows from the uniqueness in Lemma~\ref{1158} that $h_{\omega, \epsilon}^\ell=h_{\omega, \epsilon}$ for $\mathbb P$-a.e. $\omega \in \Omega$.
		Finally, the fact that $h_{\omega, \epsilon}$ spans $Y_1^{\ell, \epsilon}$ follows from Lemma~\ref{aux}.
	\end{proof}
	
	\begin{remark}
		From the previous observations, we find that for each $\epsilon \in I$ and $1\le \ell \le r-1$, the Oseledets splitting of the cocycle $(\L_{\omega, \epsilon})_{\omega \in \Omega}$ on $W^{\ell, 1}$ is given by
		\[
		W^{\ell, 1}=\text{span} \{h_{\omega, \epsilon}\} \oplus W_0^{\ell, 1},
		\]
		for $\mathbb P$-a.e $\omega \in \Omega$, where
		\[
		W_0^{\ell, 1}:=\bigg \{h\in W^{\ell, 1}: \int_{\mathbb S^1} h\, dm=0 \bigg \}.
		\]
	\end{remark}
	
	By taking into account Lemma~\ref{aux} and applying~\cite[Proposition 3.2]{BD}, we obtain the following result.
	\begin{proposition}\label{prop}
		For $\epsilon \in I$, there exist $\lambda (\epsilon) >0$ and $K^\epsilon \in \mathcal K$ such that 
		\[
		\| \L_{\omega, \epsilon}^n f\|_{W^{\ell, 1}} \le K^\epsilon (\omega)e^{-\lambda (\epsilon)n}\|f\|_{W^{\ell, 1}}, 
		\]
		for $\mathbb P$-a.e. $\omega \in \Omega$, $f\in W_0^{\ell, 1}$, $n\ge 0$ and $1\le \ell \le r-1$.
	\end{proposition}
	
	\begin{proposition}\label{1655}
		For each $1\le \ell \le r-1$, there exists $D_\ell \in \mathcal K$ such that 
		\begin{equation}\label{2220}
			\| \L_{\sigma^{-n} \omega}^n f\|_{W^{\ell, 1}} \le D_\ell (\omega)\|f \|_{W^{\ell, 1}},
		\end{equation}
		for $\mathbb P$-a.e. $\omega \in \Omega$, $n\in \N$ and $f\in W^{\ell, 1}$.
	\end{proposition}
	\begin{proof}
		Writing $K$ instead of $K^0$ and $\lambda$ instead of $\lambda (0)$, it follows from Propositions~\ref{temp} and~\ref{prop} that 
		\[
		\begin{split}
			\| \L_{\sigma^{-n} \omega}^n f\|_{W^{\ell, 1}} &=\bigg \| \L_{\sigma^{-n} \omega}^n \bigg (f- \bigg (\int_{\mathbb S^1} f\, dm \bigg )h_{\sigma^{-n} \omega, 0}+\bigg (\int_{\mathbb S^1} f\, dm \bigg )h_{\sigma^{-n} \omega, 0}\bigg )\bigg \|_{W^{\ell, 1}} \\
			&\le K(\sigma^{-n} \omega)e^{-\lambda n}\bigg \|f- \bigg (\int_{\mathbb S^1} f\, dm \bigg )h_{\sigma^{-n} \omega, 0}\bigg \|_{W^{\ell, 1}}+ \|f\|_{W^{\ell, 1}} \cdot \|h_{\omega, 0} \|_{W^{\ell, 1}} \\
			&\le K_{\lambda /2}(\omega) e^{-\frac{\lambda}{2}n} (1+ \|h_{\sigma^{-n} \omega, 0} \|_{W^{\ell, 1}})\|f\|_{W^{\ell, 1}}+ \|f\|_{W^{\ell, 1}} \cdot \|h_{\omega, 0} \|_{W^{\ell, 1}}.
		\end{split}
		\]
		Since $\omega \mapsto \|h_{\omega, 0} \|_{W^{\ell, 1}}$ is tempered, by Proposition~\ref{temp} there exists $Q\in \mathcal K$ such that 
		\[
		\|h_{\omega, 0}\|_{W^{\ell, 1}}\le Q(\omega), \quad Q(\sigma^n \omega) \le e^{\frac{\lambda}{2}|n|}Q(\omega), 
		\]
		for $\mathbb P$-a.e. $\omega \in \Omega$ and $n\in \Z$. We conclude that~\eqref{2220} holds with
		\[
		D_\ell (\omega)=K_{\lambda/2}(\omega)+K_{\lambda/2}(\omega)Q(\omega)+Q(\omega).
		\]
	\end{proof}

	\begin{lemma}\label{aux2}
		For each  $1\le \ell \le r-1$, 
		there exists $P_\ell\colon \Omega \to (0, \infty)$ measurable such that 
		\begin{equation}\label{Pl}
			\| \L_{\sigma^{-n} \omega, \epsilon}^n 1\|_{W^{\ell, 1}}\le P_\ell (\omega),
		\end{equation}
		for $\epsilon \in I$, $\mathbb P$-a.e. $\omega \in \Omega$ and $n\in \N$.
	\end{lemma}
	
	\begin{proof}
		Fix an arbitrary $1\le \ell \le r-1$. It follows from Lemma~\ref{lemma:weakLY} and the proof of~\cite[Lemma C.5]{GTQ} that  there exist $\alpha\colon \Omega \to (0, \infty)$ measurable and $D>0$\footnote{We drop the $\ell$ dependency in both $\alpha, D$ to alleviate the notation.} such that $\int_\Omega \log \alpha \, d\mathbb P<0$ and
		\begin{equation}\label{910}
			\| \L_{\omega, \epsilon} h\|_{W^{\ell,1}} \le \alpha (\omega)\|h \|_{W^{\ell,1}}+D \|h\|_{W^{\ell-1, 1}},
		\end{equation}
		for $\epsilon \in I$, $\mathbb P$-a.e. $\omega \in \Omega$ and $h\in W^{\ell, 1}$.
		By iterating~\eqref{910}, we have that 
		\begin{equation}\label{510}
			\| \L_{\sigma^{-n} \omega, \epsilon}^n 1\|_{W^{\ell,1}} \le D \bigg (1+\sum_{j=1}^n \alpha^{(j)}(\sigma^{-j} \omega) \bigg ),
		\end{equation}
		for $\epsilon \in I$, $\mathbb P$-a.e. $\omega \in \Omega$ and $n\in \N$,
		where
		\[
		\alpha^{(j)}(\sigma^{-j} \omega):=\alpha (\sigma^{-1} \omega) \cdots \alpha (\sigma^{-j} \omega).
		\]
		Note that 
		\[
		\lim_{j\to \infty} \frac 1 j \log \alpha^{(j)}(\sigma^{-j} \omega)=\lim_{j\to \infty} \frac 1 j \sum_{t=1}^j \log \alpha (\sigma^{-t} \omega) = \int_\Omega \log \alpha \, d\mathbb P<0, \quad \ \text{for $\mathbb P$-a.e. $\omega \in \Omega$.}
		\]
		Hence, there exists $\delta>0$ such that for $\mathbb P$-a.e. $\omega \in \Omega$, $\alpha^{(j)}(\sigma^{-j} \omega) \le e^{-\delta j}$ for $j\ge j(\omega)$. Set
		\[
		\bar C(\omega)=\sup_j (\alpha^{(j)}(\sigma^{-j} \omega) e^{\delta j}), \quad \omega \in \Omega.
		\]
		Now it follows readily from~\eqref{510} that~\eqref{Pl} holds with
		\[
		P_\ell(\omega):=D\bigg (1+\bar C(\omega) \frac{e^{-\delta}}{1-e^{-\delta}} \bigg), \quad \omega \in \Omega.
		\]
		
	\end{proof}
	
	\begin{lemma}\label{LL}
		For each $1\le \ell \le r-1$, we have that 
		\[
		\| h_{\omega, \epsilon}\|_{W^{\ell, 1}} \le P_\ell(\omega),
		\]
		for $\epsilon \in I$ and $\mathbb P$-a.e. $\omega \in \Omega$, where $P_\ell$ is given by Lemma~\ref{aux2}.
	\end{lemma}
	
	\begin{proof}
		Take an arbitrary $\epsilon \in I$. We claim that 
		\begin{equation}\label{h}
			h_{\omega, \epsilon}=\lim_{n\to \infty} \L_{\sigma^{-n} \omega, \epsilon}^n 1, \quad \text{for $\mathbb P$-a.e. $\omega \in \Omega$.}
		\end{equation}
		Once we establish~\eqref{h}, the conclusion of the lemma follows directly from Lemma~\ref{aux2}. Writing $K$ and $\lambda$ instead of $K^\epsilon$ and $\lambda (\epsilon)$ respectively, it follows from Proposition~\ref{prop} that 
		\begin{equation}\label{802}
			\begin{split}
				\|h_{\omega, \epsilon}-\L_{\sigma^{-n} \omega, \epsilon}^n 1\|_{W^{\ell, 1}} &=\|\L_{\sigma^{-n} \omega, \epsilon}^n(h_{\sigma^{-n} \omega, \epsilon}-1)\|_{W^{\ell, 1}} \\
				&\le K(\sigma^{-n}\omega)e^{-\lambda n} \|h_{\sigma^{-n} \omega, \epsilon}-1\|_{W^{\ell, 1}} \\
				&\le K_{\lambda/2}(\omega) e^{-\frac{\lambda}{2} n}\|h_{\sigma^{-n} \omega, \epsilon}-1\|_{W^{\ell, 1}},
			\end{split}
		\end{equation}
		for $\mathbb P$-a.e. $\omega \in \Omega$ and $n\in \N$.
		
		On the other hand, 
		since $\omega \mapsto \|h_{\omega, \epsilon}\|_{W^{\ell, 1}}$ is tempered, there exists $C\in \mathcal K$ (depending on $\epsilon$) such that 
		\begin{equation}\label{806}
			\|h_{\omega, \epsilon}-1\|_{W^{\ell, 1}} \le C(\omega) \quad \text{and} \quad C(\sigma^n \omega) \le C(\omega)e^{\frac{\lambda}{4} |n|}, 
		\end{equation}
		for $\mathbb P$-a.e. $\omega \in \Omega$ and $n\in \Z$. By~\eqref{802} and~\eqref{806}, we have that 
		\begin{equation}\label{h1}
			\|h_{\omega, \epsilon}-\L_{\sigma^{-n} \omega, \epsilon}^n 1\|_{W^{\ell, 1}} \le C(\omega)K_{\lambda/2}(\omega) e^{-\frac{\lambda}{4} n},
		\end{equation}
		for $\mathbb P$-a.e. $\omega \in \Omega$ and $n\in \N$. Letting $n\to \infty$ in~\eqref{h1} yields~\eqref{h}. The proof of the lemma is completed.
	\end{proof}

	\begin{proposition}
		Take an arbitrary $1\le \ell \le r-1$. Then, for each $s>0$ there exists a measurable $C\colon \Omega \to (0, \infty)$ such that 
		\begin{equation}\label{bid}
			\|h_{\sigma^n \omega, \epsilon}\|_{W^{\ell, 1}} \le C(\omega)e^{s |n|}, 
		\end{equation}
		for $\epsilon \in I$, $\mathbb P$-a.e. $\omega \in \Omega$ and $n\in \Z$.
	\end{proposition}
	
	\begin{proof}
		We follow closely the proof of~\cite[Lemma 4.2]{Buzzi}. 
		Since $C_\ell$ in~\eqref{eq:LY} is log-integrable, we have that $\sigma_0:=\int_{\Omega} \log C_\ell (\omega) \, d\mathbb P(\omega)$ is finite. Without any loss of generality, we may assume that 
		$\frac{s}{2(\sigma_0+1)} <1$. By applying Birkhoff's ergodic theorem, we have that there exists $\Omega_2 \subset \Omega$, $\mathbb P(\Omega_2)>1-\frac{s}{4(\sigma_0+1)}$ such that for $\omega \in \Omega_2$, sufficiently large $n$ and each choice of the sign $\pm$ we have that 
		\[
		\sum_{0\le j<n} \log C_{\ell} (\sigma^{\pm j} \omega) \le (\sigma_0+1)n.
		\]
		Let $P_\ell$ be given by Lemma~\ref{aux2}. Note that there  exists $A>1$ such that $\mathbb P(\{ \omega \in \Omega: P_\ell (\omega)>A\})<\frac{s}{4(\sigma_0+1)}$. By applying Birkhoff's ergodic theorem once again, we have that for $\mathbb P$-a.e. $\omega \in \Omega$ and each choice of $\pm$, 
		\[
		\lim_{k\to \pm \infty}\frac{1}{|k|} \# \{0\le j<|k|: P_\ell (\omega) >A \ \text{or} \ \sigma^{\pm j}\omega \notin \Omega_2 \} <\frac{s}{2(\sigma_0+1)}<1.
		\]
		By using~\eqref{eq:LY}, Lemma~\ref{LL} and arguing exactly as in the proof of~\cite[Lemma 4.2]{Buzzi},  we find that there exists $N_0\colon \Omega \to \N$ such that for each $\epsilon \in I$ and $\mathbb P$-a.e. $\omega \in \Omega$,
		\[
		\|h_{\sigma^n \omega, \epsilon}\|_{W^{\ell, 1}}  \le Ae^{s|n|} \quad \text{for $n\in \Z$, $|n|\ge  N_0(\omega)$.}
		\]
		By writing $h_{\sigma^n \omega, \epsilon}=\L_{\sigma^{-N_0(\omega)} \omega}^{n+N_0(\omega)}h_{\sigma^{-N_0(\omega)} \omega, \epsilon}$ for $|n|< N_0(\omega)$ and
		invoking~\eqref{eq:LY} again, we  conclude that~\eqref{bid} holds with 
		\[
		C(\omega)=Ae^{2sN_0(\omega)} \prod_{i=-N_0(\omega)}^{N_0(\omega)-2}\max \{C_\ell (\sigma^i \omega), 1 \}.
		\]
		The proof of the proposition is completed. 
	\end{proof}
	
	As a direct consequence of the previous result, we obtain the following:
	\begin{cor}\label{KORO}
		Let $(\epsilon_k)_{k\in \N}$ be an arbitrary sequence in $I$. Then, for each $1\le \ell \le r-1$, $\omega \mapsto \sup_k \|h_{\omega, \epsilon_k}\|_{W^{\ell, 1}}$ is a tempered random variable. 
	\end{cor}
	
	\subsection{Estimates in the triple norm}
	Let $\mathbf T\colon \Omega \to C^r(I\times \mathbb S^1, \mathbb S^1)$ be an arbitrary parameterized smooth expanding on average cocycle. 
	We continue to use the same notation as in the previous subsection. 
	In addition, we will write 
	$T_{\omega}$ and $\L_\omega$ instead of $T_{\omega, 0}$ and $\L_{\omega, 0}$ respectively. 
	
	\begin{lemma}\label{LE1}
		There exists $Q_1\in \mathcal K$ such that 
		\[
		\|( \L_{\omega, \epsilon}-\L_\omega) f\|_{L^1} \le Q_1(\omega) d_{C^1}(T_{\omega, \epsilon}, T_\omega)\|f\|_{W^{1,1}},
		\]
		for $\epsilon \in I$, $\mathbb P$-a.e. $\omega \in \Omega$ and $f\in W^{1,1}$.
	\end{lemma}
	\begin{proof}
		For every $f\in W^{1,1}$, $\varphi \in C^1(\mathbb S^1)$,  $\epsilon \in I$ and $\omega \in \Omega$, we have that 
		\[
		\int_{\mathbb S^1} (\mathcal L_{\omega, \epsilon}f-\mathcal L_\omega f)\varphi \, dm=\int_{\mathbb S^1} f(\varphi \circ T_{\omega, \epsilon}-\varphi \circ T_\omega)\, dm.
		\]
		Set $\Phi_{\omega, \epsilon}(x):=\frac{1}{T_\omega'(x)} \int_{T_\omega x}^{T_{\omega, \epsilon} x}\varphi(z)\, dm(z)$. Then,
		\[
		\Phi_{\omega, \epsilon}'(x)=-\frac{1}{T_\omega'(x)}T_\omega''(x)\Phi_{\omega, \epsilon}(x)+\frac{T_{\omega, \epsilon}'(x)}{T_\omega'(x)}\varphi (T_{\omega, \epsilon} x)-\varphi (T_\omega x).
		\]
		Hence substituting for $\varphi(T_{\omega}x)$ and integrating by parts,
		\[
		\begin{split}
			\int_{\mathbb S^1} (\mathcal L_{\omega, \epsilon}f-\mathcal L_\omega f)\varphi \, dm &=\int_{\mathbb S^1} f\Phi_{\omega, \epsilon}' \, dm+\int_{\mathbb S^1} f \bigg [\frac{T_\omega''}{T_\omega'}\Phi_{\omega, \epsilon}+\bigg (1-\frac{T_{\omega, \epsilon}'}{T_\omega'} \bigg )\varphi \circ T_{\omega, \epsilon} \bigg ]\, dm \\
			&=-\int_{\mathbb S^1} f'\Phi_{\omega, \epsilon} \, dm+\int_{\mathbb S^1} f \bigg [\frac{T_\omega''}{T_\omega'}\Phi_{\omega, \epsilon}+\bigg (1-\frac{T_{\omega, \epsilon}'}{T_\omega'} \bigg )\varphi \circ T_{\omega, \epsilon} \bigg ]\, dm.
		\end{split}
		\]
		Observe that
		\[ |\Phi_{\omega, \epsilon} |_\infty \le \underline{\lambda}(\omega)^{-1} |\varphi |_\infty d_{C^1} (T_{\omega, \epsilon}, T_\omega) \quad \text{and} \quad
		\bigg |1-\frac{T_{\omega, \epsilon}'}{T_\omega'}\bigg |_\infty \le \underline{\lambda}(\omega)^{-1}d_{C^1}(T_{\omega, \epsilon}, T_\omega). \]
		Moreover, \eqref{n1} implies that 
		\[
		\begin{split}
			\bigg | \frac{T_\omega''}{T_\omega'}\Phi_{\omega, \epsilon}\bigg |_\infty &\le \underline{\lambda}(\omega)^{-1}K(\omega)  |\Phi_{\omega, \epsilon}|_\infty 
			\le \underline{\lambda}(\omega)^{-2} K(\omega) d_{C^1} (T_{\omega, \epsilon}, T_\omega )| \varphi |_\infty.
		\end{split}
		\]
		Hence, 
		\[
		\begin{split}
			\bigg | \int_{\mathbb S^1} (\mathcal L_{\omega, \epsilon}f-\mathcal L_\omega f)\varphi \, dm \bigg | & \le \underline{\lambda}(\omega)^{-1}  \| f'\|_{L^1} |\varphi |_\infty d_{C^1} (T_{\omega, \epsilon}, T_\omega) \\
			&\phantom{=}+ \underline{\lambda}(\omega)^{-2} K(\omega)  \| f\|_{L^1} d_{C^1} (T_{\omega, \epsilon}, T_\omega )| \varphi |_\infty  \\
			&\phantom{\le}+\|f\|_{L^1}|\varphi |_\infty \underline{\lambda}(\omega)^{-1}d_{C^1}(T_{\omega, \epsilon}, T_\omega),
		\end{split}
		\]
		which implies that the desired conclusion holds with
		\[
		Q_1(\omega)=\underline{\lambda}(\omega)^{-1}+\underline{\lambda}(\omega)^{-2} K(\omega),
		\]
		which due to log-integrability of $\underline{\lambda}$ and $K$ belongs to $\mathcal K$.
	\end{proof}
	The proof of the following lemma is inspired by  the proof of~\cite[Proposition 5.11]{HC}.
	\begin{lemma}\label{LE2}
		For each $\ell \in \N$ such that $2\le \ell$ and $\ell+1\le r$, there exists $Q_{\ell} \in \mathcal K$ such that 
		\[
		\|(\L_{\omega, \epsilon}-\L_\omega)f\|_{W^{\ell-1,1}}\le Q_\ell(\omega)\|f\|_{W^{\ell,1}}d_{C^{\ell}}(T_{\omega,\epsilon},T_\omega),
		\]
		for $\epsilon \in I$, $\mathbb P$-a.e. $\omega \in \Omega$ and $f\in W^{\ell, 1}$.
	\end{lemma}
	
	\begin{proof}
		Throughout the proof, $C$ will denote a generic element of $\mathcal K$ (depending on $\ell$ but not on $\epsilon$) that can change values from one occurrence to the next.
		
		By Lemma \ref{lemma:Crim}, we have that  for $g\in L^\infty(\mathbb S^1)$,
		\begin{align*}
			&\int_{\mathbb S^1} ((\L_{\omega, \epsilon}-\L_\omega)f)^{(\ell-1)}g \, dm \\
			&=\int_{\mathbb S^1} \L_{\omega, \epsilon}\left(\sum_{j=0}^{\ell-1}(T'_{\omega, \epsilon})^{-2(\ell-1)} G_{\ell-1,j}(T_{\omega, \epsilon}',\dots,T_{\omega, \epsilon}^{(\ell)})f^{(j)}\right)g\, dm\\
			&\phantom{=}-\int_{\mathbb S^1} \L_\omega\left(\sum_{j=0}^{\ell-1}(T'_{\omega})^{-2(\ell-1)}G_{\ell-1,j}(T_{\omega}',\dots,T_{\omega}^{(\ell)})f^{(j)}\right)g\, dm\\
			&=\int_{\mathbb S^1}(\L_{\omega, \epsilon}-\L_\omega)\left(\sum_{j=0}^{\ell-1}(T'_{\omega, \epsilon})^{-2(\ell-1)} G_{\ell-1,j}(T_{\omega, \epsilon}',\dots,T_{\omega, \epsilon}^{(\ell)})f^{(j)}\right)g \, dm\\
			&\phantom{=}+\int_{\mathbb S^1}\sum_{j=0}^{\ell-1}\left(\frac{G_{\ell-1,j}(T_{\omega, \epsilon}',\dots,T_{\omega, \epsilon}^{(\ell)})}{(T'_{\omega, \epsilon})^{2(\ell-1)}}-\frac{G_{\ell-1,j}(T_{\omega}',\dots,T_{\omega}^{(\ell)})}{(T'_{\omega})^{2(\ell-1)}}\right)f^{(j)}g\circ T_\omega \,  dm\\
			&=:(I)+(II).
		\end{align*}
		By Lemma~\ref{LE1}, we have (using that $\|\psi h\|_{W^{1,1}} \le \|\psi \|_{C^1} \cdot \|h\|_{W^{1,1}}$ for $\psi \in C^1(\mathbb S^1)$ and $h\in W^{1,1}$) that
		
		\begin{equation}\label{r0}
			\begin{split}
				&\bigg |\int_{\mathbb S^1}(\L_{\omega, \epsilon}-\L_\omega)\left(\sum_{j=0}^{\ell-1}(T'_{\omega, \epsilon})^{-2(\ell-1)} G_{\ell-1,j}(T_{\omega, \epsilon}',\dots,T_{\omega, \epsilon}^{(\ell)})f^{(j)}\right)g \, dm\bigg | \\
				&\le C(\omega)d_{C^1}(T_{\omega,\epsilon},T_\omega)|g|_{\infty}\left\|\sum_{j=0}^{\ell-1}(T'_{\omega, \epsilon})^{-2(\ell-1)}G_{\ell-1,j}(T_{\omega, \epsilon}',\dots,T_{\omega, \epsilon}^{(\ell)})f^{(j)} \right\|_{W^{1,1}}\\
				&\le C(\omega)d_{C^1}(T_{\omega,\epsilon},T_\omega)|g|_{\infty}\|f\|_{W^{\ell,1}}\sum_{j=0}^{\ell-1}\|(T'_{\omega, \epsilon})^{-2(\ell-1)}G_{\ell-1,j}(T_{\omega, \epsilon}',\dots,T_{\omega, \epsilon}^{(\ell)})\|_{C^1}\\
				&\le C(\omega)d_{C^1}(T_{\omega,\epsilon},T_\omega)|g|_{\infty}\|f\|_{W^{\ell,1}}\sum_{j=0}^{\ell-1}\|(T'_{\omega, \epsilon})^{-2(\ell-1)} \|_{C^1} \cdot \|G_{\ell-1,j}(T_{\omega, \epsilon}',\dots,T_{\omega, \epsilon}^{(\ell)})\|_{C^1}.
			\end{split}
		\end{equation}

		Moreover, \eqref{n1} implies that 
		\[
		\|(T'_{\omega, \epsilon})^{-2(\ell-1)} \|_{C^1} \le \underline{\lambda}(\omega)^{-2(\ell-1)}+2(l-1)\underline{\lambda}(\omega)^{-2\ell+1}K(\omega).
		\]
		By taking into account~\eqref{der}, we conclude that 
		\begin{equation}\label{r1}
			\begin{split}
				&\bigg |\int_{\mathbb S^1}(\L_{\omega, \epsilon}-\L_\omega)\left(\sum_{j=0}^{\ell-1}(T'_{\omega, \epsilon})^{-2(\ell-1)} G_{\ell-1,j}(T_{\omega, \epsilon}',\dots,T_{\omega, \epsilon}^{(\ell)})f^{(j)}\right)g \, dm\bigg | \\
				&\le C(\omega)d_{C^1}(T_{\omega,\epsilon},T_\omega)|g|_{\infty}\|f\|_{W^{\ell,1}},
			\end{split}
		\end{equation}
		for $\epsilon \in I$, $\mathbb P$-a.e. $\omega \in \Omega$ and $f\in W^{\ell, 1}$.
		
		Similarly, 
		\begin{equation}\label{r2}
			\begin{split}
				& \bigg | \int_{\mathbb S^1}\sum_{j=0}^{\ell-1}\left(\frac{G_{\ell-1,j}(T_{\omega, \epsilon}',\dots,T_{\omega, \epsilon}^{(\ell)})}{(T'_{\omega, \epsilon})^{2(\ell-1)}}-\frac{G_{\ell-1,j}(T_{\omega}',\dots,T_{\omega}^{(\ell)})}{(T'_{\omega})^{2(\ell-1)}}\right)f^{(j)}g\circ T_\omega \,  dm \bigg | \\
				&\le |g|_\infty \sum_{j=0}^{\ell -1}\bigg |\frac{G_{\ell-1,j}(T_{\omega, \epsilon}',\dots,T_{\omega, \epsilon}^{(\ell)})}{(T'_{\omega, \epsilon})^{2(\ell-1)}}-\frac{G_{\ell-1,j}(T_{\omega}',\dots,T_{\omega}^{(\ell)})}{(T'_{\omega})^{2(\ell-1)}}\bigg |_\infty \cdot \|f^{(j)} \|_{L^1} \\
				&\le |g|_\infty \underline{\lambda}(\omega)^{-4(l-1)}\|f\|_{W^{\ell, 1}} \sum_{j=0}^{\ell -1}P_{\omega, \epsilon, j, \ell} \\
				&\le C(\omega)|g|_\infty \|f\|_{W^{\ell, 1}},
			\end{split}
		\end{equation}
		for $\epsilon \in I$, $\mathbb P$-a.e. $\omega \in \Omega$ and $f\in W^{\ell, 1}$, where
		\[
		P_{\omega, \epsilon, j, \ell}=\bigg|G_{\ell-1,j}(T_{\omega, \epsilon}',\dots,T_{\omega, \epsilon}^{(\ell)})(T'_{\omega})^{2(\ell-1)}-G_{\ell-1,j}(T_{\omega}',\dots,T_{\omega}^{(\ell)})(T'_{\omega, \epsilon})^{2(\ell-1)}\bigg |_\infty.
		\]
		The conclusion of the lemma follows readily from~\eqref{r0}, \eqref{r1} and~\eqref{r2}.
	\end{proof}
	The following result is a simple consequence of~\eqref{n2}, Lemmas~\ref{LE1} and~\ref{LE2} and the mean-value theorem.
	\begin{cor}\label{KK}
		For each $\ell \in \N$, $\ell+1\le r$ there exists $\tilde Q_l\in \mathcal K$ such that 
		\[
		\|(\L_{\omega, \epsilon}-\L_\omega)f\|_{W^{\ell-1,1}}\le |\epsilon| \tilde Q_\ell(\omega) \|f\|_{W^{\ell,1}},
		\]
		for $\epsilon \in I$, $\mathbb P$-a.e. $\omega \in \Omega$ and $f\in W^{\ell, 1}$, where $\|\cdot \|_{W^{0,1}}:=\| \cdot \|_{L^1}$.
	\end{cor}

	\subsection{Derivative operator}
	We again take an  arbitrary parameterized smooth expanding on average cocycle $\mathbf T\colon \Omega \to C^r(I\times \mathbb S^1, \mathbb S^1)$.
	We start this section by recalling some elementary facts. For $f\in C^\ell(\mathbb S^1,\R)$, we set $M_f(\phi):=f\phi$. Then, $M_f$ is a bounded operator on $W^{\ell,1}(\mathbb S^1,\R)$.  More precisely, there exists a constant $C$ (depending only on $\ell$) such that 
	\begin{equation}\label{eq:prodop}
		\|M_f(\phi)\|_{W^{\ell,1}}=\|f\phi \|_{W^{\ell, 1}} \le C\|f\|_{C^\ell}\cdot \|\phi\|_{W^{\ell,1}}, \quad \text{for $\phi \in W^{\ell, 1}$.}
	\end{equation}
	For $\phi\in C^r(\mathbb S^1,\R)$, set 
	\begin{align}
		g_{\omega,\epsilon}&:=\frac{1}{|T'_{\omega,\epsilon}|}\in C^{r-1}(\mathbb S^1,\R)\\
		V_{\omega,\epsilon}(\phi) &:=-\frac{\phi'}{T'_{\omega,\epsilon}}\cdot\partial_\epsilon T_{\omega,\epsilon}\in C^{r-1}(\mathbb S^1,\R).
	\end{align}
	Remark that $\phi \mapsto V_{\omega,\epsilon}(\phi)$ defines a bounded operator from $W^{\ell,1}$ to $W^{\ell-1,1}$, for any $\ell\le r$. 
	\\We also define
	\begin{equation}
		J_{\omega,\epsilon}:=\frac{\partial_\epsilon g_{\omega,\epsilon}+V_{\omega,\epsilon}(g_{\omega,\epsilon})}{g_{\omega,\epsilon}}\in C^{r-2}(\mathbb S^1,\R).
	\end{equation} 
	A direct consequence of Definition \ref{EC}, equations  \eqref{n1} and \eqref{n2} is that there exists $\tilde K \in \mathcal{K} $ such that for $\epsilon \in I$ and $\mathbb P$-a.e. $\omega \in \Omega$,
	\begin{equation}\label{eq:goodbound}
		\|g_{\omega,\epsilon}\|_{C^{r-1}}\le \tilde K(\omega), \quad  \|\partial_{\epsilon} g_{\omega,\epsilon}\|_{C^{r-2}} \le \tilde K(\omega) \quad \text{and} \quad \|J_{\omega,\epsilon}\|_{C^{r-2}}\le \tilde K(\omega).
	\end{equation}
	We note that $\tilde K$ is a polynomial in the tempered random variable $K$ appearing in \eqref{n1} and \eqref{n2}.
	As before, let $ \L_{\omega, \epsilon}$ be the transfer operator associated to $T_{\omega, \epsilon}:=\mathbf T(\omega)(\epsilon, \cdot)$. By formal differentiation (see also \cite[p.18]{DS1}), we have that
	\begin{align*}
		\partial_\epsilon [\L_{\omega,\epsilon}\phi]&=\L_{\omega,\epsilon}(J_{\omega,\epsilon}\phi+V_{\omega,\epsilon}\phi)\\
		\partial_{\epsilon}^2[\L_{\omega, \epsilon}\phi]&= \L_{{\omega,\epsilon}}\left(J_{\omega,\epsilon}^2\phi+J_{\omega,\epsilon}(V_{\omega,\epsilon}\phi)+V_{\omega,\epsilon}(J_{\omega,\epsilon}\phi)+V_{\omega,\epsilon}(V_{\omega,\epsilon}\phi)+[\partial_\epsilon J_{\omega,\epsilon}]\cdot \phi+\partial_{\epsilon}[V_{\omega,\epsilon}\phi]\right).
	\end{align*}
	
	For later usage we will need the differentiation in absence of the perturbation. Namely, we define $\hat{\L}_\omega$ by
	\begin{equation}\label{hatL}
		\hat{\L}_\omega \phi:=\L_\omega(J_{\omega, 0} \phi+V_{\omega, 0}\phi).
	\end{equation}
	By Lemma \ref{lemma:weakLY}, we have, for $1\le j\le r-3$,
	\[
	\begin{split}
		&\|\partial_\epsilon^2[\mathcal L_{{\omega,\epsilon}}\phi] \|_{W^{j,1}}  \\
		&= \| \mathcal L_{{\omega,\epsilon}}\left(J_{\omega,\epsilon}^2\phi+J_{\omega,\epsilon}(V_{\omega,\epsilon}\phi)+V_{\omega,\epsilon}(J_{\omega,\epsilon}\phi)+V_{\omega,\epsilon}(V_{\omega,\epsilon}\phi)+[\partial_\epsilon J_{\omega,\epsilon}]\cdot \phi+\partial_{\epsilon}[V_{\omega,\epsilon}\phi]\right)\|_{W^{j,1}}\\
		&\le C_j(\omega) \|J_{\omega,\epsilon}^2\phi+J_{\omega,\epsilon}(V_{\omega,\epsilon}\phi)+V_{\omega,\epsilon}(J_{\omega,\epsilon}\phi)+V_{\omega,\epsilon}(V_{\omega,\epsilon}\phi)+[\partial_\epsilon J_{\omega,\epsilon}]\cdot \phi+\partial_{\epsilon}[V_{\omega,\epsilon}\phi]\|_{W^{j,1}}.
	\end{split}
	\]
	We will now estimate the  $W^{j,1}$ norm of each term appearing on the R.H.S above. By $C$ we will denote a generic constant (that doesn't depend on $\omega$, $\epsilon$ or $\phi$). Firstly, we have (see \eqref{eq:prodop} and \eqref{eq:goodbound})
	\[
	\|J_{\omega,\epsilon}^2\phi\|_{W^{j,1}} \le C\|J_{\omega, \epsilon}\|_{C^j} \| J_{\omega, \epsilon}\phi \|_{W^{j,1}} \le C\|J_{\omega, \epsilon}\|_{C^j}^2 \| \phi \|_{W^{j,1}}\le C\tilde K^2(\omega)\| \phi \|_{W^{j,1}}.
	\]
	Secondly,
	\[
	\begin{split}
		\|J_{\omega,\epsilon}(V_{\omega,\epsilon}\phi)\|_{W^{j,1}} & \le C\|J_{\omega, \epsilon}\|_{C^j}\cdot \|\phi'g_{\omega, \epsilon} \partial_{\epsilon}T_{\omega, \epsilon}(\cdot)\|_{W^{j,1}}  \\
		&\le C\|J_{\omega, \epsilon}\|_{C^j}\cdot \|g_{\omega, \epsilon} \partial_{\epsilon}T_{\omega, \epsilon}(\cdot)\|_{C^j} \cdot \|\phi\|_{W^{j+1,1}}\\
		&\le C \tilde K^2(\omega)K(\omega)\|\phi\|_{W^{j+1,1}}.
	\end{split}
	\]
	Thirdly,
	\[
	\begin{split}
		\|V_{\omega,\epsilon}(J_{\omega,\epsilon}\phi)\|_{W^{j,1}} &\le C\|g_{\omega, \epsilon} \partial_{\epsilon}T_{\omega, \epsilon}(\cdot)\|_{C^j} \cdot \|(J_{\omega,\epsilon}\phi)'\|_{W^{j,1}} \\
		&\le C\|g_{\omega, \epsilon} \partial_{\epsilon}T_{\omega, \epsilon}(\cdot)\|_{C^j}\|J_{\omega,\epsilon}\phi\|_{W^{j+1,1}} \\
		&\le C\|g_{\omega, \epsilon} \partial_{\epsilon}T_{\omega, \epsilon}(\cdot)\|_{C^j}\| J_{\omega, \epsilon}\|_{C^{j+1}}\|\phi\|_{W^{j+1,1}}\\
		&\le C \tilde K^2(\omega)K(\omega)\|\phi\|_{W^{j+1,1}}.
	\end{split}
	\]
	Fourthly,
	\[
	\begin{split}
		\|V_{\omega,\epsilon}(V_{\omega,\epsilon}\phi)\|_{W^{j,1}} &\le C\|g_{\omega, \epsilon} \partial_{\epsilon}T_{\omega, \epsilon}(\cdot)\|_{C^j} \cdot \|(V_{\omega,\epsilon}\phi)'\|_{W^{j,1}} \\
		&\le C\|g_{\omega, \epsilon} \partial_{\epsilon}T_{\omega, \epsilon}(\cdot)\|_{C^j}\|V_{\omega,\epsilon}\phi \|_{W^{j+1,1}} \\
		&\le C\|g_{\omega, \epsilon} \partial_{\epsilon}T_{\omega, \epsilon}(\cdot)\|_{C^j}\|g_{\omega, \epsilon} \partial_{\epsilon}T_{\omega, \epsilon}(\cdot)\|_{C^{j+1}} \cdot \|\phi '\|_{W^{j+1,1}} \\
		&\le C\|g_{\omega, \epsilon} \partial_{\epsilon}T_{\omega, \epsilon}(\cdot)\|_{C^j}\|g_{\omega, \epsilon} \partial_{\epsilon}T_{\omega, \epsilon}(\cdot)\|_{C^{j+1}} \cdot \|\phi \|_{W^{j+2,1}}\\
		&\le C\tilde K^2(\omega)K^2(\omega)\|\phi \|_{W^{j+2,1}}.
	\end{split}
	\]
	Fifthly,
	\[
	\begin{split}
		\|[\partial_\epsilon J_{\omega,\epsilon}]\cdot \phi\|_{W^{j,1}} &\le C \|\partial_\epsilon J_{\omega,\epsilon}\|_{C^j} \cdot \|\phi \|_{W^{j,1}}.
	\end{split}
	\]
	Noting that
	\[
	\begin{split}
		\partial_\epsilon J_{\omega,\epsilon} &=\partial_\epsilon T_{\omega, \epsilon}' \bigg (\partial_\epsilon g_{\omega, \epsilon}-g_{\omega, \epsilon}' g_{\omega, \epsilon} \partial_{\epsilon}T_{\omega, \epsilon}\bigg )\\ 
		&\phantom{=} + T_{\omega, \epsilon}'\bigg (\partial_\epsilon^2 g_{\omega, \epsilon}-\partial_\epsilon g_{\omega, \epsilon}' g_{\omega, \epsilon} \partial_{\epsilon}T_{\omega, \epsilon}-g_{\omega, \epsilon}' \partial_\epsilon g_{\omega, \epsilon} \partial_{\epsilon}T_{\omega, \epsilon}-g_{\omega, \epsilon}'   g_ {\omega, \epsilon} \partial_{\epsilon}^2 T_{\omega, \epsilon}\bigg ),
	\end{split}
	\]
	we obtain that 
	\[
	\|[\partial_\epsilon J_{\omega,\epsilon}]\cdot \phi\|_{W^{j,1}}\le C\bar K(\omega)\|\phi \|_{W^{j,1}},
	\]
	where $\bar K$ is a tempered random variable, which is a  polynomial in the tempered random variable $K$ appearing in \eqref{n1} and~\eqref{n2}.
	\\ Finally, 
	\[
	\begin{split}
		\|\partial_{\epsilon}[V_{\omega,\epsilon}\phi]\|_{W^{j,1}} &=\|\phi' \partial_\epsilon g_{\omega, \epsilon}\partial_{\epsilon}T_{\omega, \epsilon}+\phi' g_{\omega, \epsilon} \partial_\epsilon^2 T_{\omega, \epsilon}\|_{W^{j,1}} \\
		&\le C\left(\|\partial_\epsilon g_{\omega, \epsilon}\partial_{\epsilon}T_{\omega, \epsilon} \|_{C^j}+
		\|g_{\omega, \epsilon} \partial_\epsilon^2 T_{\omega, \epsilon}\|_{C^j}\right)\|\phi'\|_{W^{j,1}} \\
		&\le C \left(\|\partial_\epsilon g_{\omega, \epsilon}\partial_{\epsilon}T_{\omega, \epsilon} \|_{C^j}+
		\|g_{\omega, \epsilon} \partial_\epsilon^2 T_{\omega, \epsilon}\|_{C^j}\right) \| \phi \|_{W^{j+1,1}}\\
		&\le 2C\tilde K(\omega)K(\omega) \| \phi \|_{W^{j+1,1}}.
	\end{split}
	\]
	
	Putting the last six estimates together, we conclude that for each $1\le j\le r-3$,  there exists $\tilde C_j\in \mathcal K$ such that
	\[
	\|\partial_\epsilon^2[\mathcal L_{\omega,\epsilon}\phi] \|_{W^{j,1}} \le \tilde C_j(\omega)\| \phi \|_{W^{j+2,1}}.
	\]
	We can now proceed as in~\cite[p. 18]{DS1}:  by Taylor's formula of order two, we have that 
	\[
	\begin{split}
		\|\mathcal L_{\omega, \epsilon}\phi-\L_{\omega}\phi -\epsilon \hat{\mathcal L}_\omega \phi \|_{W^{j,1}} &=\left\| \int_0^\epsilon \int_0^\eta \partial_\epsilon^2[\mathcal L_{\omega,\epsilon}\phi ]\rvert_{\epsilon=\xi}\, d\xi \, d\eta \right\|_{W^{j,1}} \\
		& \le \frac{ \epsilon^2}{ 2}\tilde C_j(\omega)\| \phi \|_{W^{j+2,1}}.
	\end{split}
	\]
	
	Hence,  for $1\le j\le r-3$, $\epsilon \in I\setminus \{0\}$, $\phi \in W^{j+2, 1}$ and $\mathbb P$-a.e. $\omega \in \Omega$,
	\begin{equation}\label{1700}
		\bigg \| \frac{1}{\epsilon} \bigg (\mathcal L_{\omega, \epsilon}\phi-\mathcal L_\omega \phi \bigg )-\hat{\L}_\omega \phi  \bigg \|_{W^{j,1}} \le \frac{|\epsilon|}{2} \tilde{C}_j(\omega) \|\phi \|_{W^{j+2,1}}.
	\end{equation}

	\subsection{Proof of Theorem \ref{LR1}}\label{sec:LRproof}

	In order to establish Theorem~\ref{LR1}, we
		 apply Theorem~\ref{LR} for the choice of spaces
		\[
		\mathcal B_w=W^{1,1}, \quad \mathcal B_s=W^{2,1}, \quad   \mathcal B_{ss}=W^{3,1},
		\]
		and with the functional $\psi$ given by $\psi (h)=\int_{\mathbb S^1}h\, dm$.
		It remains to observe the following:
		\begin{itemize}
			\item \eqref{c11} follows from Corollary~\ref{KORO};
			\item \eqref{c22} is established in Proposition~\ref{1655};
			\item \eqref{c33} is a consequence of Proposition~\ref{prop} (applied for $\epsilon=0$);
			\item \eqref{c44} and~\eqref{c55} are proved in Corollary~\ref{KK};
			\item \eqref{c66} is established in~\eqref{1700}, where $\hat \L_\omega$ is given by~\eqref{hatL}.
		\end{itemize}
	
	\appendix
	\section{An example where quenched response holds and annealed response fails}
	Let us remind the reader that whenever quenched linear response holds (see e.g. Theorem \ref{LR1}), one has, for any smooth observable $\phi$, that 
	\begin{align*}
	\int_{\mathbb S^1}\phi \hat h_\omega dm &= \sum_{n=0}^{\infty}\int_{\mathbb S^1}\phi\cdot\L^{n}_{\sigma^{-n}\omega}\hat{\L}_{\sigma^{-n-1}\omega} h_{\sigma^{-n-1}\omega} dm\\
	&= \sum_{n=0}^{\infty}\int_{\mathbb S^1}\phi\circ T^n_{\sigma^{-n}\omega}\cdot\hat{\L}_{\sigma^{-n-1}\omega} h_{\sigma^{-n-1}\omega} dm.
	\end{align*}
	In particular, the integral on the L.H.S is well-defined. \\If one is interested in \emph{annealed} linear response, the object  of interest becomes the (double) integral $\int_{\Omega}\int_{\mathbb S^1}\phi \hat h_\omega dm~d\mathds P$: when it holds this integral is well-defined, and one has
	\[\int_\Omega\int_{\mathbb S^1}\phi \hat h_\omega dm~d\mathds P = \sum_{n=0}^{\infty}\int_{\Omega}\int_{\mathbb S^1}\phi\circ T_\omega^n\cdot\hat{\L}_{\sigma^{-1}\omega} h_{\sigma^{-1}\omega} dm~d\mathds P.\]
	Given that quenched results concerns a.e trajectory, and that annealed one are on average, in general one expects that a quenched result implies its annealed counterpart. This intuition is validated in the context of response theory by the existing results \cite{DS,RS,CN}. However, when there is no stochastic uniformity, this intuition can fail, as the next example shows:  
	\medskip
	
	\noindent Consider $(\tilde\Omega,\mathcal B,\mathds Q, S)$ the full-shift over $\{1,2,\dots\}$, with probability vector \\$(Z,Z/2^{2+\delta},\dots,Z/n^{2+\delta},\dots)$, for some $0\le\delta\le 1$, $Z$ being the normalization constant.
	\\Let $h:\tilde\Omega\to\R$ be the (positive) observable defined by $h(\omega)=\omega_0$ if $\omega:=(\omega_n)_{n\in \mathbb Z} \in\tilde\Omega$. Note that
	\[\int_{\tilde\Omega}h~d\mathds Q=\sum_{i\ge 1}i\cdot\frac{Z}{i^{2+\delta}}=Z\sum_{i=1}^\infty\frac{1}{i^{1+\delta}}<+\infty,\]
	when $0<\delta\le 1$.
	\\Define $(\Omega,\mathcal F,\mathbb P,\sigma)$ to be the suspension over $S$ with roof function $h$, i.e. $\Omega:=\{(\omega,i)\in\tilde\Omega\times\N,~ 0\le i<h(\omega)\}$, $\sigma:\Omega\circlearrowleft$ is given by
	\[
	\sigma(\omega,i):=\left\{
	\begin{aligned}
		&(\omega,i+1)~&\text{if}~i<h(\omega)-1\\
		&(S\omega,0)~&\text{if}~i=h(\omega)-1
	\end{aligned}
	\right.
	\]
	and $\mathbb P(A):=\left(\int_{\tilde\Omega} h~ d\mathds Q\right)^{-1}\sum_{i\ge 0}\mathds Q\left(A\cap(\tilde\Omega\times\{i\})\right)$.
	\\We can now define our random system: take $T_0:\mathbb S^1\circlearrowleft$ to be the doubling map, i.e. $T_0(x)=2x \ (\text{mod} \ 1)$, and let $T_1$ be the identity map on $\mathbb S^1$.
	We consider the random circle map $T_{(\omega,i)}$, $(\omega,i)\in\Omega$ defined by
	\[
	T_{(\omega,i)}:=\begin{cases}
		T_1 & \text{if $i<h(\omega)-1$} \\
		T_0 & \text{if  $i=h(\omega)-1$.}
	\end{cases}
	\]
Clearly, 	$(T_{(\omega,i)})_{(\omega, i)\in \Omega}$ is an expanding on average cocycle, i.e.  \[	\int_{\Omega}\log\lambda_{(\omega, i)}~d\mathds P ((\omega, i))>0.\] 
 Indeed, observe that  $(T_{(\omega,i)})_{(\omega, i)\in \Omega}$  is a particular case of cocycles $(T_\omega)_{\omega}$ introduced in Example~\ref{EXAMP}.
	Moreover, $\mu_{(\omega, i)}=m$ for $(\omega, i) \in \Omega$, where $m$ denotes the Lebesgue measure on $\mathbb S^1$.
	 
	Let $n_c$ be the (random) covering time for the interval $[0,1/2]$:
	\[
	n_c (\omega, i):=\min \{ k\in \N: T_{(\omega,i)}^k([0,1/2])=\mathbb S^1 \}, \quad (\omega, i) \in \Omega. 
	\]
	Then, it is easy to see that 
$n_c(\omega, i)=h(\omega)-i$. Observe that $n_c$ is not integrable. Indeed, for each $N\in \N$ we have that 
	\begin{align*}
		\mathbb P(n_c(\omega,i)=N)&=\left(\int_{\tilde\Omega} h~ d\mathds Q\right)^{-1}\sum_{i\ge 0}\mathds Q(h(\omega)-i=N)\\
		&=\left(\int_{\tilde\Omega} h~ d\mathds Q\right)^{-1}\sum_{i\ge N}\mathds Q(\omega_0=i)\\
		&=\left(\int_{\tilde\Omega} h~ d\mathds Q\right)^{-1}\sum_{i\ge N}\frac{Z}{i^{2+\delta}}\sim\frac{C}{N^{1+\delta}},
	\end{align*}
	for some constant $C>0$, which easily implies that $n_c$ is not integrable. 
	\\We now introduce the observable: consider a $\psi\in C^\infty(\mathbb S^1)$, such that:
	\begin{itemize}
		\item [(i)] $\text{supp}~\psi\subset I$, where $I\subset\mathbb S^1$ is an interval such that $I\cap\dfrac{1}{2}I=\emptyset$ (in particular, we can take any small enough $I\not\ni 0$).
		\item [(ii)] $\int_{\mathbb S^1}\psi~dm=0$ and $\int_{\mathbb S^1}\psi^2~dm=1$.
	\end{itemize}
	It is easy to see that  $\int_{\mathbb S^1}\psi\cdot\psi\circ T_0~dm=0$. Therefore, we have
	\[
	\int_{\mathbb S^1} \psi\cdot\psi\circ T_{(\omega,i)}^n~dm=
	\left\{
	\begin{aligned}
		&1~\text{if}~n<n_c(\omega,i)\\
		&0~\text{otherwise.}
	\end{aligned}
	\right.
	\]
In particular, it follows from the non-integrability of $n_c$ that 
\begin{equation}\label{eq:nonint}
	\int_{\Omega}\sum_{n=0}^\infty\int_{\mathbb S^1} \psi\cdot\psi\circ T_{(\omega,i)}^n~dm~d\mathds P =+\infty
	\end{equation}
	Let us now introduce the perturbed cocycle $T_{(\omega,i),\epsilon}:= D_{\epsilon}\circ T_{(\omega,i)}$, where $D_{\epsilon}:\mathbb S^1\circlearrowleft$ is given by
	\[D_{\epsilon}(x):=x-\epsilon \int_0^x\psi~dm \ (mod \ 1).\]
	Denote by $S(x):=-\int_0^x\psi~dm$. Since this perturbation is smooth and deterministic ($D_\epsilon$ does not depend on $\omega$), we may apply Theorem \ref{LR1}, and obtain that for a.e. $(\omega,i)\in\Omega$, the integral $\int_{\mathbb S^1}\psi \hat h_{(\omega, i)} dm$ is well defined, and satisfies
	\begin{align*}
		\int_{\mathbb S^1}\psi \hat h_{(\omega,i)} dm &= \sum_{n=0}^{\infty}\int_{\mathbb S^1}\psi\L_{\sigma^{-n}(\omega,i)}^n\hat{\L}_{\sigma^{-n-1}(\omega,i)}h_{\sigma^{-n-1}(\omega,i)} dm\\
		& = -\sum_{n=0}^\infty \int_{\mathbb S^1}\psi\circ T_{\sigma^{-n}(\omega,i)}^{n} S' dm\\
		& =\sum_{n=0}^\infty \int_{\mathbb S^1}\psi\circ T_{\sigma^{-n}(\omega,i)}^{n} \psi dm,
\end{align*} 
where we used that  $\hat{\L}_{(\omega,i)}f = -(\L_{(\omega,i)}(f)\cdot S)'$ (which can be proved by arguing as in \cite[Sec 6.2]{GS}) and $h_{(\omega,i)}=\mathds 1$ for our particular example.

Hence, from \eqref{eq:nonint}
\begin{equation}
	\int_{\Omega}\int_{\mathbb S^1}\psi \hat h_{\omega,i} dm d\mathds P=+\infty.
\end{equation}
In particular, annealed response cannot hold.
	\section*{Acknowledgments}
	\par\noindent D.D. was supported in part by Croatian Science Foundation under the project IP-2019-04-1239 and by the University of Rijeka under the projects uniri-prirod-18-9 and uniri-pr-prirod-19-16. 
	\\J.S. was supported by the European Research Council (E.R.C.) under the European Union's Horizon 2020 research and innovation programme (grant agreement No 787304).
	\\P.G was partially supported by the PRIN Grant 2017S35EHN  "Regular and stochastic behaviour in dynamical systems" and by the INDAM - GNFM 2020 grant “Deterministic and stochastic dynamical systems for climate studies”.
	
	\bibliographystyle{amsplain}
	\begin{footnotesize}
	
	\end{footnotesize}
\end{document}